\documentclass[a4paper,11pt]{article}
\usepackage[T1]{fontenc}
\usepackage[utf8x]{inputenc} 
\usepackage{lmodern}
\usepackage{amsmath}
\usepackage{amsthm}
\usepackage{amssymb}
\usepackage{authblk}
\usepackage{graphicx}
\usepackage{enumitem}
\usepackage{xcolor}
\usepackage{hyperref}

\usepackage{amsthm}
\usepackage{amssymb}
\usepackage{xparse}
\usepackage{thmtools}
\usepackage{stackrel}
\usepackage{bbold}
\usepackage{relsize}
\usepackage{tikz}
\usepackage{dsfont}
\usepackage{bigints}

\DeclareFontFamily{OMX}{lmex}{}
\DeclareFontShape{OMX}{lmex}{m}{n}{<-> lmex10}{}  

\usetikzlibrary{hobby}

\newcommand{\R}{\mathbb{R}}
\newcommand{\C}{\mathbb{C}}
\newcommand{\Z}{\mathbb{Z}}
\newcommand{\N}{\mathbb{N}}

\newcommand{\dd}{\mathrm{d}}

\renewcommand{\O}{\mathcal{O}}

\DeclareMathOperator{\tr}{tr}

\DeclareMathOperator{\dist}{dist}

\makeatletter
\newtheoremstyle{indented}
{7pt} 
{7pt} 
{} 
{1.5em} 
{\bfseries} 
{.} 
{.5em} 
{} 

\theoremstyle{definition}

\newtheorem{defn}{Definition}[section]

\theoremstyle{plain}
\newtheorem*{theorem*}{Theorem}

\newtheorem{theorem}{Theorem}

\newtheorem{prop}[defn]{Proposition}

\newtheorem{lem}[defn]{Lemma}
\newtheorem{conj}{Conjecture}

\theoremstyle{definition}
\newtheorem{rem}[defn]{Remark} 

\makeatletter
\renewcommand*\env@matrix[1][*\c@MaxMatrixCols c]{%
  \hskip -\arraycolsep
  \let\@ifnextchar\new@ifnextchar
  \array{#1}}
\makeatother

\usepackage{a4wide}
\title{Real-analytic geodesics in the Mabuchi space of Kähler metrics
  and quantization}
\author{Alix
  Deleporte\thanks{alix.deleporte@universite-paris-saclay.fr},
  \framebox{Steve
  Zelditch}\thanks{Steve Zelditch was
    affiliated to Northwestern University, at the department of
    Mathematics, until his death on Sept. 11, 2022. The first version
    of this article
    was mostly finished at this date.}}
\affil{$^*$Universit\'e Paris-Saclay, CNRS, Laboratoire de math\'ematiques
  d'Orsay, F-91405, Orsay, France.
}

\newcommand{\todo}[1]{}

\newcommand{\old}[1]{}

\usepackage{scrtime}

\usetikzlibrary{decorations.pathreplacing}
\usetikzlibrary{decorations.markings}

\begin{document}

\maketitle

\begin{abstract}
  We prove the convergence of quantized Bergman geodesics to the
  Mabuchi geodesics for the initial value problem, in the case of
  real-analytic initial data and in short time. This partially solves
  a conjecture of Y. Rubinstein and the last author. We also argue
  against the existence of a solution to the boundary value problem,
  generically in real-analytic regularity.

  To this end, we introduce non-self-adjoint Fourier
      Integral Operators, and prove that they are satisfactory
    approximations to the Bergman geodesics, that is, solutions of a semiclassical
    Schrödinger equation with skew-adjoint Hamiltonian.
\end{abstract}

\section{Setting and main results}
\label{sec:setting}

\subsection{The Mabuchi metric}
\label{sec:mabuchi-metric}

Let $(M,J,\omega_0)$ be a Kähler manifold of complex dimension $d$. We let $\mathcal{H}$ denote
the space of $C^{1,1}$ changes of Kähler potentials on $(M,J,\omega_0)$, that is, the following
open subset of $C^{1,1}(M,\R)$:
\[
  \mathcal{H}=\{\phi\in C^{1,1}(M,\R),\omega_{\phi}:=\omega_0+i\partial\overline{\partial}\phi>0\}.\]
 We endow the infinite-dimensional space $\mathcal{H}$ with a Riemannian
metric named after Mabuchi 
\cite{mabuchi_symplectic_1987,semmes_complex_1992,donaldson_symmetric_1999}: at
$(\phi,v)\in T\mathcal{H}\simeq
\mathcal{H}\times C^{1,1}(M,\R)$, the squared norm of $v$ is
\begin{equation}
  \label{eq:Mabuchi-metric}
  \int_M|v|^2\omega_{\phi}^{\wedge d}.
\end{equation}
The geodesic equation associated to the Mabuchi metric is
\begin{equation}
  \label{eq:Mabuchi_geod_eq}
  \ddot{\phi}(t)=|\partial\dot{\phi}(t)|_t^2,
\end{equation}
where $|\cdot|_t$ is the norm on $TM$ induced by the Kähler metric
$(M,J,\omega_{\phi(t)})$ and $\partial$ is the holomorphic gradient.

The space $\mathcal{H}$ is negatively curved, and it is in fact
non-positive in the sense of Alexandrov \cite{calabi_space_2002}. If $v_1,v_2$ are normalised elements of
$T\mathcal{H}$ over the same base point $\phi$, the curvature element is
\begin{equation}
  \label{eq:curv-Mabuchi}
  K(v_1,v_2)=-\frac{1}{4}\int_M
  |\{v_1,v_2\}_{\phi}|^2\omega_{\phi}^{\wedge d}.
\end{equation}

The initial value problem associated with the PDE \eqref{eq:Mabuchi_geod_eq}
is, generally speaking, ill-posed. In this article, we will be concerned with
\emph{real-analytic} initial data, for which the conclusion of the
Cauchy-Kovalevskaya theorem applies \cite{mabuchi_symplectic_1987} so that
a local in time (but not global in time in general) solution exists. We
stress that this is the only situation in which a local existence
theorem is known; conversely, local existence is known to be false in $C^3$ \cite{rubinstein_cauchy_2017}.

Beginning with \cite{phong_monge-ampere_2006,phong_test_2007} there
has been a particular focus on the relationship between geodesics
in $\mathcal{H}$ and geodesics in the approximating spaces
$\mathcal{B}_k$ of \emph{Bergman metrics} of degree $k$, defined
  in
  Section \ref{sec:berez-toepl-quant}. The spaces $\mathcal{B}_k$ are
symmetric spaces of the form $SL(d_k, \C)/SU(d_k)$, where
the dimension $d_k$ tends to infinity with $k$, and they have their own geodesics, known
as Bergman rays. Two natural maps ${\rm Hilb}_k:\mathcal{H}\to
\mathcal{B}_k$ and ${\rm FS}_k:\mathcal{B}_k\to \mathcal{H}$ allow to
compare Mabuchi geodesics and Bergman rays with each other.

Since $SL(d_k,\C)/SU(d_k)$ is finite-dimensional,
Bergman rays can always be extended in infinite time. Every Bergman ray corresponds to a subsolution of
\eqref{eq:Mabuchi_geod_eq}, that is, for every $k\in \N$, if $A:\R\to \mathcal{B}_k$ is a
  geodesic and $\phi_k(t)={\rm FS}_k(A(t))$, then 
\cite{berndtsson_positivity_2009}
  \begin{equation}\label{eq:subsol}
    \ddot{\phi}_k(t)\geq  |\partial \dot{\phi}_k(t)|^2_t.
  \end{equation}
This result provides a natural way to construct
a subsolution of the initial value problem as the limsup of the
Bergman rays as $k\to +\infty$. This subsolution is always defined in infinite time,
which raises the question of its relationship with actual solutions
and motivates the study of Bergman metrics,

It was conjectured in \cite{rubinstein_cauchy_2011} that this subsolution
matches the actual solution of the initial value problem for \eqref{eq:Mabuchi_geod_eq} as long
as the latter exists; this conjecture also appears, in a different
form, in \cite{bhattacharyya_exponential_2022}. The main results of
this article (Theorems \ref{thr:closeness_geodesics_IVP} and \ref{thr:FIOs}) are a partial proof of this
claim, in the real-analytic case, by showing convergence for
small times if $\omega_0$ and $\dot{\phi}(0)$ are real-analytic and
describing the asymptotic structure of these geodesics as $k\to +\infty$.

\subsection{Holomorphic sections of line bundles}
\label{sec:berez-toepl-quant}
Suppose that $(M,J,\omega_0)$ is \emph{polarised}, that is,
$\omega_0\in H^2(M,2\pi\Z)$. There exists a (not
  necessarily unique) Hermitian line bundle
$(L,h_0)\to M$ with curvature $-i\omega_0$, called a prequantum line bundle.

For every $k\in \N$, the space of holomorphic sections
$H^0(M,L^{\otimes k})$ inherits a Hilbert space structure from the Hermitian metric
$h_0$ on $L$:
\[
  \|u\|_{0}^2:=\int_M \|u(x)\|_{h_0^{\otimes k}}^2\omega_0^{\wedge d}[\dd x].
\]
In this way, $H^0(M,L^{\otimes k})$ sits as a finite-dimensional
subspace of $L^2(M,L^{\otimes k})$. The associated orthogonal
projector $\Pi_k:L^2(M,L^{\otimes k})\to H^0(M,L^{\otimes k})$ is
called the Bergman projector.

Let $\phi\in \mathcal{H}$. The symplectic structure
\[
  \omega_{\phi}=\omega_0+i\partial\overline{\partial}\phi
  \]
belongs to the same cohomology class as $\omega_0$; if we perform the same construction with $\omega_{\phi}$ instead of
$\omega_0$, we find the same topological bundle $L$ with a different
Hermitian structure:
\[
  \|\cdot\|^2_{h_{\phi}}:=\|\cdot\|^2_{h_0}e^{-\phi}.
\]

In this way, one can naturally quantize an element of $\mathcal{H}$ into a
sequence of Hilbert space structures on the spaces
$H^0(M,L^{\otimes k})$, where the squared norm of an element $u$ is now
\begin{equation}\label{eq:Hilb}
\|u\|_{\phi}^2=\int_M\|u(x)\|_{h_{\phi}^{\otimes
    k}}^2\omega_{\phi}^{\wedge d}[dx].
\end{equation}
The new Kähler
structure $\phi$ then allows to define a new Bergman projector
$\Pi_k^{\phi}$.

We let $\mathcal{B}_k$ denote the set of
Hilbert structures (scalar products) on $H^0(M,L^{\otimes k})$, and we call ${\rm
  Hilb}_k$ the map $\phi\mapsto \|\cdot\|_{\phi}^2$.

Conversely, elements of $\mathcal{B}_k$ allow to define Kähler
structures on $M$ via a map ${\rm FS}_k$
defined as follows: let $H\in \mathcal{B}_k$ and let $(s_j)$ be an
orthonormal basis of $H$, then we define
\[
  {\rm FS}_k(H):x\mapsto \frac 1k \log \sum_j\|s_j(x)\|^2_{h_0}-d\frac{\log(k)}{k}.
\]

${\rm FS}_k$ is asymptotically a left
inverse to ${\rm Hilb}_k$; see Proposition \ref{prop:Tian-Zelditch}
for details.\footnote{The usual definition of ${\rm FS}_k$ does not
  contain the asymptotically vanishing $-d\log(k)/k$
  term. We choose this definition to compensate for the universal
  subleading term in ${\rm FS}_k({\rm Hilb}_k(\phi))$ in the usual definition.}

Given a ``reference'' element $H_0\in \mathcal{B}_k$, the set
  $\mathcal{B}_k$ can be identified with the space $S^{++}(H_0)$ of positive
  definite $H_0$-self-adjoint operators via the formula
  \begin{equation}\label{eq:S++}
    \langle u,v\rangle_H=\langle u,Av\rangle_{H_0}.
  \end{equation}
  Indeed, any $A\in S^{++}(H_0)$ defines a new Hilbert structure $H$ in
  this way, and from $H$ one uniquely recovers $A$ -- for instance, choosing an orthonormal basis $(e_j)$ for $H_0$,
  the matrix coefficients of $A$ are $(\langle
  e_j,e_k\rangle_{H})_{j,k}$. This identification allows to
  understand the tangent space $T_{H_0}\mathcal{B}_k$ as the space of
  $H_0$-self-adjoint operators on $H^0(M,L^{\otimes k})$. A natural
  Riemannian metric on $\mathcal{B}_k$ is then given by prescribing
  the norm of an element $(H,A)\in T\mathcal{B}_k$ as the
  Hilbert--Schmidt norm of $A$ as a $H$-self-adjoint operator:
\begin{equation}\label{eq:normBk}
  \|A\|_H^2=\tr(A^2).
\end{equation}

Geodesics, for this metric, are one-parameter families $\gamma(t)$ of
scalar products, such that 
\[
  \langle \cdot,\cdot\rangle_{\gamma(t)}=\langle \cdot,e^{tA}\cdot\rangle_{\gamma(0)},
\]
where $A$ is a $\gamma(0)$-self-adjoint operator. The curvature element
between normalised elements $A_1,A_2$ of $T\mathcal{B}_k$ over the same
base point is
\begin{equation}\label{eq:curvBk}
  K(A_1,A_2)=\frac 14\tr([A_1,A_2]^2).
\end{equation}
It is already somewhat pleasant to compare this formula with \eqref{eq:curv-Mabuchi}.

Given
$(\phi,v)\in T\mathcal{H}$, one can use \eqref{eq:Hilb} to compute
$\dd{\rm Hilb}_k(v)$: indeed, as $\epsilon\in \R$ tends to zero, given
$u_1,u_2\in H^0(M,L^{\otimes k})$, 
\begin{align*}
  \langle u_1,u_2\rangle_{{\rm Hilb}_k(\phi+\epsilon v)}&=\int_M\langle
  u_1(x),u_2(x)\rangle_{h_{\phi}^{\otimes k}}e^{-\epsilon k
  v}(\omega_{\phi}+i\epsilon \partial\overline{\partial}v)^{\wedge
                                                        d}\\
  &=\int_M\langle
  u_1(x),u_2(x)\rangle_{h_{\phi}^{\otimes
    k}}\left(\omega_{\phi}^{\wedge d}-\epsilon k v
    \omega_{\phi}^{\wedge d}+i\epsilon d \partial
    \overline{\partial}v\wedge \omega_{\phi}^{\wedge
    d-1}+O_k(\epsilon^2)\right)\\
  &=\langle u_1,u_2\rangle_{{\rm Hilb}_k(\phi)}+\epsilon\int_M\langle
  u_1(x),u_2(x)\rangle_{h_{\phi}^{\otimes
    k}}\left(-kv+\Delta_{\phi}v\right)\omega_{\phi}^{\wedge
    d}+O_k(\epsilon^2)\\
  =&\langle u_1,u_2\rangle_{{\rm Hilb}_k(\phi)}+\epsilon\langle
  u_1,\Pi_k^{\phi}[(-kv+\Delta_{\phi}v)u_2]\rangle_{{\rm Hilb}_k(\phi)}+O_k(\epsilon^2).
\end{align*}Introducing \emph{Berezin--Toeplitz operators}
\begin{equation}\label{eq:contra_Toep}
  L^{\infty}(M,\R)\ni v \to T_k^{\phi}(v):=\Pi_k^{\phi}v\Pi_k^{\phi}
\end{equation}
which are ${\rm Hilb}_k(\phi)$-self-adjoint operators, we obtain
\begin{equation}\label{eq:dHilbk}
  \dd{\rm Hilb}_k(v)= T_k^{\phi}\left(-kv+\Delta_{\phi}v)\right).
\end{equation}

The map ${\rm Hilb}_k$ between the infinite-dimensional space
$\mathcal{H}$ and the finite-dimensional space
$\mathcal{B}_k$ cannot be injective. However, it
approximately preserves, at the infinitesimal level, the Riemannian data (up to a
rescaling) thanks to the classical-quantum correspondence for
Berezin--Toeplitz operators.
\begin{prop}[\cite{bordemann_toeplitz_1994},\cite{charles_berezin-toeplitz_2003}]\label{prop:quantum_class_corresp}
  Let $\phi\in \mathcal{H}$ be such that
  $\omega_{\phi}\in C^{\infty}$ and let $v_1,v_2\in
  C^{\infty}(M,\R)$. Then
  \[
    \tr(T_k^{\phi}(v_1)^2)=\frac{k^d}{\pi^d}\int |v_1|^2\omega_{\phi}^{\wedge
      d}+O(k^{d-1}).
    \]
  \[
    \tr([T_k^{\phi}(v_1),T_k^{\phi}(v_2)]^2)=-\frac{k^{d-2}}{\pi^{d}}\int
    |\{v_1,v_2\}_{\omega_{\phi}}|^2\omega_\phi^{\wedge
      d}+O(k^{d-3}).\]
\end{prop}
Thus, up to a scale factor $C_dk^{\frac d2 + 1}$, the geometry
of $\mathcal{B}_k$ is presumed to reflect that of
$\mathcal{H}$: indeed, after using equation \eqref{eq:dHilbk}, the
norms \eqref{eq:Mabuchi-metric}\eqref{eq:normBk} and the curvature elements
\eqref{eq:curv-Mabuchi}\eqref{eq:curvBk} match up to a relative error $O(k^{-1})$. Using this fact, it was proved that the image by
${\rm Hilb}_k$ of geodesics almost solve the geodesic equation.

\begin{prop}\label{prop:sung3.5}[\cite{chen_space_2012}, Proposition
  3.5] Let $t\mapsto \phi(t)$ denote a geodesic path of smooth Kähler structures in $\mathcal{H}$. For
  every $t$ in the domain of this path, let
  $c_k(t)={\rm Hilb}_k(\phi(t))$. Then, as $k\to +\infty$, the curve $c_k$
  almost satisfies the geodesic equation on $\mathcal{B}_k$:
  \[
    \|\nabla_{\dot{c}_k}\dot{c}_k\|=o(k^{\frac d2 + 1}).
    \]
  \end{prop}
 In this article, we study the
initial value problem before and after application of ${\rm Hilb}_k$, and prove that the
distance between the projected geodesic and the actual geodesic is
small. Unfortunately, because the spaces $\mathcal{H}$ and
$\mathcal{B}_k$ are negatively curved with very large or unbounded curvature, one cannot apply a Grönwall-type
 lemma, in the spirit of \cite{jost_riemannian_2008}, Corollary 4.6.1,
 to prove this claim using only Proposition
 \ref{prop:sung3.5}. Nevertheless, under hypotheses of
 real-analyticity, we are able to prove closeness of the two geodesics for short times.

 \begin{theorem}\label{thr:closeness_geodesics_IVP}
   Suppose that $\omega_0$ is
  real-analytic and let $\dot{\phi}_0\in C^{\omega}(M,\R)$. Let
  $\phi(t)$ denote the geodesic in $\mathcal{H}$ with initial value
  $(0,\dot{\phi}_0)$, which is well-defined for short time. For all $k\in
  \N$, let
  \[
    c_k:t\mapsto {\rm Hilb}_k(\phi(t)).
  \]
  Let also $\gamma_k(t)$ be the geodesic on $\mathcal{B}_k$ with
  initial value $({\rm Hilb}_k(0),\dd {\rm Hilb}_k(\dot{\phi}_0))=(c_k(0),\dot{c}_k(0))$.
  Then there exists $t_0>0$ and $C>0$ such that, uniformly on $t\in [0,t_0]$, as
  $k\to +\infty$, one has
  \begin{align}
    \dist_{\mathcal{B}_k}(c_k(t),\gamma_k(t))&\leq Ck^{\frac
                                               d2} \label{eq:close-geods-Bk} \\
    \label{eq:close-geods-FSk}
    \dist_{\mathcal{H}}(\phi(t),{\rm FS}_k(\gamma_k(t)))&\leq C k^{-1}.
  \end{align}
\end{theorem}
Recalling from Proposition \ref{prop:quantum_class_corresp} that the
natural scaling between the distances on $\mathcal{B}_k$ and
$\mathcal{H}$ is $k^{-\frac d2 - 1}$, so that
 \eqref{eq:close-geods-Bk} and \eqref{eq:close-geods-FSk} express the
same magnitude of distances between the considered objects.
 


Theorem \ref{thr:closeness_geodesics_IVP} is a 
consequence of a more technical result about the \emph{integral kernels} of
$c_k(t)$ and $\gamma_k(t)$ which is interesting in its own right and
which we describe now. Recall from \eqref{eq:S++} that both
$c_k(t)$ and $\gamma_k(t)$ can be cast as positive symmetric operators
on $H^0(M,L^{\otimes k})$ endowed with the scalar product ${\rm
  Hilb}_k(\phi(0))$. We prove that $c_k(t)$ and $\gamma_k(t)$ have
integral kernels that are relatively close to
each other.

\begin{theorem}\label{thr:FIOs}
  Suppose that $\omega_0$ is
  real-analytic and let $\dot{\phi}_0\in C^{\omega}(M,\R)$. Let
  $\phi(t)$ denote the geodesic in $\mathcal{H}$ with initial value
  $(0,\dot{\phi}_0)$, which is well-defined for short time and let
  $c_k(t)={\rm Hilb}_k(\phi(t))$. Let
  $\gamma_k(t)$ be the geodesic on $\mathcal{B}_k$ with initial value
  $({\rm Hilb}_k(0),{\rm d}{\rm Hilb}_k(\dot{\phi}_0))$. Then there
  exists
  \begin{itemize}
  \item a time $t_0>0$
  \item an open neighbourhood $V$ of the diagonal of $M\times
    M$
  \item for each $t\in [-t_0,t_0]$, a non-vanishing holomorphic section $\Phi(t)\in
    L\boxtimes \overline{L}$ over $V$, with analytic dependence on $t$ (see
    Definition \ref{def:boxtimes} of the line bundle $L\boxtimes \overline{L}$)
  \item two uniformly bounded sequences of real-analytic functions 
    $(a_k(t))_{k\in \N},(b_k(t))_{k\in \N}$ on $V$
  \item constants $C>0$, $c>0$
  \end{itemize}
  such that for every $t\in [-t_0,t_0]$, for every $u,v\in
  H^0(M,L^{\otimes k})$, one has
  \begin{align}
    \label{eq:ck_FIO_intro}
    \left|\langle u,v\rangle_{c_k(t)}-k^d\int_V \langle \Phi(t,x,y)^{\otimes
    k},\overline{u(x)}\otimes v(y)\rangle_{h_0^{\otimes k}\times
    \overline{h_0}^{\otimes k}}a_k(t,x,y)\omega_0^{\wedge d}(\dd x)
    \omega_0^{\wedge d}(\dd y)\right|&\leq Ce^{-ck}\|u\|_{c_k(t)}\|v\|_{c_k(t)}\\ \label{eq:gammak_FIO_intro}
   \left|\langle u,v\rangle_{\gamma_k(t)}-k^d\int_V \langle \Phi(t,x,y)^{\otimes
    k},\overline{u(x)}\otimes v(y)\rangle_{h_0^{\otimes k}\times
    \overline{h_0}^{\otimes k}}b_k(t,x,y)\omega_0^{\wedge d}(\dd x)
    \omega_0^{\wedge d}(\dd y)\right|&\leq Ce^{-ck}\|u\|_{c_k(t)}\|v\|_{c_k(t)}.
  \end{align}
\end{theorem}

The proof of Theorem
\ref{thr:FIOs}, and of the link with Theorem \ref{thr:closeness_geodesics_IVP}, necessitates a calculus (approximate composition and
inversion) of families of operators on $H^0(M,L^{\otimes k})$ with
integral kernels (in the sense given by \eqref{eq:ck_FIO_intro} or \eqref{eq:gammak_FIO_intro}) of the form
\[
  (x,y)\mapsto k^d\Phi(x,y)^{\otimes k}a_k(x,y),
\]
which we will call \emph{analytic Fourier Integral Operators}. We
develop this theory in Section \ref{sec:analyt-four-integr}.

Equation \eqref{eq:ck_FIO_intro} boils down to the structure of the
Szeg\H{o} projector $\Pi_k^{\phi}$, which in the analytic case is
now well-known to exponential precision \cite{rouby_analytic_2018,deleporte_toeplitz_2018,hezari_off-diagonal_2017,charles_analytic_2021,deleporte_direct_2022}. Equation \eqref{eq:gammak_FIO_intro} is new: it
concerns the integral kernel of an operator of the form
$e^{tkT_k(v)}$, where $v=-\dot{\phi}(0)+\frac{1}{k}\Delta
\dot{\phi}(0)$. It is known that the \emph{imaginary time} kernel
$e^{itkT_k(v)}$ has the structure of a complex Fourier Integral operator
\cite{zelditch_pointwise_2018,charles_quantum_2020}, at least up to $O(k^{-\infty})$; in the limit $k\to +\infty$
one recovers the Hamilton flow of $-\dot{\phi}(0)$, but usual
($C^{\infty}$) techniques are limited to the analysis of unitary
operators. A semiclassical analysis of the kernel of non-unitary
propagators was only previously known in
the case of quadratic symbols \cite{hormander_symplectic_1995,pravda-starov_generalized_2018}.

Formula \eqref{eq:gammak_FIO_intro} is expected to generalise to
propagators associated with more
general non-self-adjoint Berezin--Toeplitz operators, with, hopefully,
interesting consequences for the description of the dynamics and the spectrum
of these operators.

\subsection{The boundary value problem}
\label{sec:boundary_value}

The boundary value problem for \eqref{eq:Mabuchi_geod_eq}, namely the
problem of finding a geodesic on $\mathcal{H}$ with fixed endpoints,
is seemingly better behaved than the initial value problem. In fact, this
problem is formally equivalent to an elliptic nonlinear boundary value problem,
the Homogeneous Complex Monge-Ampere (HCMA) equation
\cite{donaldson_symmetric_1999}. It was progressively proved \cite{chen_space_2000,blocki_geodesics_2012,chu_c11_2017} that
any two points in $\mathcal{H}$ (i.e. $C^{1,1}$ Kähler metrics in the
same cohomology class) are joined by a unique shortest length
geodesic of $C^{1,1}$ metrics (in the HCMA sense).
Weak convergence of the Bergman
rays (with projected boundary values by ${\rm Hilb}_k$) to the
geodesic was proved in \cite{phong_monge-ampere_2006} and strong convergence in
\cite{chen_space_2012}. As a byproduct, the distances and angles on
$\mathcal{H}$ are asymptotically preserved by the map
${\rm Hilb}_k$. In these results, the main ingredient from asymptotic
analysis is the study of the properties of the kernel of the Bergman
projector $\Pi_k^{\phi}$ on and the diagonal as $k\to +\infty$
(Propositions \ref{prop:Tian-Zelditch} and \ref{prop:calc_Toeplitz}).

As the regularity increases, the boundary value problem becomes harder
to solve. Generically, two $C^k$ metrics can be joined by a
$C^{\frac{3k}{4}+1}$ geodesic for $k\geq 5$ \cite{chen_geodesic_2020}, but there exist
smooth metrics which cannot be joined by a $C^2$ geodesic
\cite{lempert_geodesics_2013} and there exist arbitrarily close
analytic metrics which cannot be joined by a smooth geodesic
\cite{hu_obstacle_2021}. It is conjectured in
\cite{chen_geodesic_2020} that smooth metrics close to each other are
generically joined by a smooth metric. Because of the link with the
space of Hamiltonian diffeomorphisms, we make a conjecture in the
opposite direction for analytic metrics.

\begin{conj}\label{conj:no_analytic_bd_value}
  Let $E$ and $E'$ be two analytic function spaces on $M$ (i.e. Banach
  spaces of analytic functions containing all analytic functions with
  sufficiently small radius of injectivity).

  Then there exists an open dense subset of
  $(E\cap \mathcal{H})^2$, in which no pair of elements is linked by a geodesic in $E'$.
\end{conj}
This conjecture has for immediate consequence (by
choosing a countable sequence of analytic function spaces $E'$
containing all analytic functions) that there exists a countable
intersection of open dense sets in $(E\cap \mathcal{H})^2$ of points not
linked to each other by an analytic geodesic. In this sense, the
opposite of the conjecture in \cite{chen_geodesic_2020} would hold for real-analytic metrics.

Because of the apparent loss of a fraction $\frac{1}{4}$ of
derivatives appearing in \cite{chen_geodesic_2020}, we conjecture that
generic (in some sense) analytic metrics are linked by a
$\frac 43$-Gevrey geodesic.

\subsection{Techniques and perspectives}
\label{sec:techniques}

We emphasize that there are many formal arguments indicating that
Theorems \ref{thr:closeness_geodesics_IVP} and \ref{thr:FIOs} should be true. The essential problem in this article is that the
formal arguments are based on approximate propagators, or parametrices. It is often
straightforward to prove a somewhat formal convergence for these approximations, but the
degree of precision of the parametrix, even in the smooth case, is not
enough. More precisely, we wish to obtain a good description of the
integral kernel of the Bergman geodesic
\[
  U(t,x,y)=e^{tkT_k(-v+k^{-1}\Delta v)}(x,y).
\]
This is the analytic continuation in $t$ of the Schrödinger propagator
\[
  U(i\tau,x,y)=e^{i\tau kT_k(-v+k^{-1}\Delta v)}(x,y),
\]
which is well-understood as a \emph{Fourier Integral Operator}
\cite{zelditch_pointwise_2018,charles_quantum_2020} if the initial
data is smooth. However, the precision of this description is
$O(k^{-\infty})$, whereas exponential precision $O(e^{-ck})$, for some
$c>0$, is needed for our purposes. One reason for this is the Duhamel
formula, associated with the fact that for $t\in \R$ one has
\[
  \log\|U(t)\|_{H^0\to H^0}\sim_{k\to +\infty} k\max(t\sup(v),t\inf(v)).
\]
An exponential level of precision is only
available in special cases, including (for short times) in the
real-analytic case, using the recently developed framework of
Berezin--Toeplitz quantization in real-analytic regularity
\cite{rouby_analytic_2018,deleporte_toeplitz_2018,hezari_off-diagonal_2017,charles_analytic_2021,deleporte_direct_2022}.

Our proof of Theorem \ref{thr:FIOs} uses a representation of
both $c_k(t)$ and $\gamma_k(t)$ as Fourier Integral Operators
with complex, real-analytic phase. The crucial point is
that, in this representation, both $c_k(t)$ and
$\gamma_k(t)$ have the same canonical relation (the same $\Phi$). In particular, we
interpret a real-analytic change of Kähler structure $\phi$ on $M$ as a
biholomorphism $\mathcal{L}$ between
neighborhoods of $M$ in its complexification; a path of Kähler
structures $\phi(t)$ is a path of biholomorphisms $\mathcal{L}(t)$. The image of a geodesic for the Mabuchi metric is
then a piece of one-parameter subgroup of biholomorphisms. This
interpretation (which, formally, stems from Proposition \ref{prop:quantum_class_corresp}) was one of the motivations for the introduction of
Berezin--Toeplitz quantization in the treatment of the Mabuchi
problem.

Conjecture \ref{conj:no_analytic_bd_value} is also linked to the
interpretation of changes of Kähler structures as complex
symplectomorphisms. Geodesics correspond to autonomous Hamiltonian
flows, and to support this claim, we use the
fact that, among real Hamiltonian diffeomorphisms, the autonomous ones
are non-generic.

The link between the geometry of Mabuchi space and that of Hamiltonian
diffeomorphisms (where geodesics are autonomous flows) uses a complexification argument. Thus, it is not
surprising that, as the regularity of the data increases, the
behaviours of the two problems become closer.

This link is also interesting from the perspective of optimal
transport. Mabuchi geodesics, that is, solutions of the complex
Monge-Ampère equation, are instances of time-dependent, $\omega$-preserving maps from the
complexification of $M$ to itself; they form a totally real
subspace of this space $\mathcal{H}^{\C}$, and a complementary subspace is formed by
$\omega$-preserving maps from $M$ to itself. This mimics the usual
optimal transport situation where optimal transport maps, solving a
real Monge-Ampère equation, are ``transverse'' to
solutions of the incompressible Euler equation inside a larger space of maps. This parallel suggests
several relevant questions: is there a ``polar decomposition'' where a
general map is the composition of a change of Kähler metric and a
real Hamiltonian diffeomorphism?
Since the boundary value problem for the geodesic equation among real
Hamiltonian differomorphisms is not well-posed, can one suitably
generalise the problem in order to find a minimal length geodesic?
Some of these questions were already addressed in
\cite{donaldson_holomorphic_2002}; in real-analytic regularity we hope
to be able to develop this theory.

In Section \ref{sec:analyt-toepl-oper} we recall the basic ingredients
from (complexified) Kähler geometry and semiclassical analysis in
real-analytic regularity which we will need. We use in particular the
structure of the Bergman kernel in real-analytic regularity
(Proposition \ref{prop:calc_Toeplitz}). Section
\ref{sec:analyt-four-integr} develops basic tools for the treatment of
non-unitary Fourier Integral Operators. One can compose and invert
such operators, and change the reference Kähler metric with respect to
which they are defined, as long as all involved objects are close to
the ``classical'' case of the Bergman kernel. Then in Section
\ref{sec:gronw-lemma-appr} we use these tools to prove our main
claims; we notably prove in Proposition \ref{prop:propag_is_FIO} that
quantum propagators of skew-adjoints operators are Fourier Integral
Operators.

We conclude this introduction with a remark about positivity and canonical bundles:
letting $K$ be the canonical bundle over $M$, another convention for
Berezin--Toeplitz operators and the ${\rm Hilb}_k$ and ${\rm FS}_k$
maps, is to consider holomorphic sections of $L^{\otimes k}\otimes
K$ rather than $L^{\otimes k}$. In terms of semiclassical analysis, the difference between the two
cases is subprincipal, and although the
particular identities (notably Proposition \ref{prop:calc_Toep_subp})
that we use here should be modified, our main three claims should also
hold in the case of a twist by the canonical bundle (or, in fact, any
fixed line bundle). A notable difference, and the reason why we focus
on the untwisted case, is that the property that Bergman geodesics are
subsolutions of the Mabuchi equation (equation \eqref{eq:subsol}) is specific to the untwisted
case. In fact, there is a reverse inequality in the canonical twist case,
where the image by ${\rm Hilb}_k$ of a Mabuchi geodesic is a
sub-$\mathcal{B}_k^{\rm twist}$
geodesic \cite{berndtsson_positivity_2009}. Our end goal is to study
natural candidates for Mabuchi geodesics after explosion, and
unfortunately the liminf of the twisted Bergman geodesics is
not, in general, a super-solution of \eqref{eq:Mabuchi-metric}. 


\section{Semiclassical analysis of Berezin-Toeplitz operators}
\label{sec:analyt-toepl-oper}

\subsection{Polarisation, complexification, and
  complex symplectic geometry}
\label{sec:polar-compl-assoc}

Recall that $(M,J,\omega_0)$ is a real-analytic Kähler manifold: $J$
is a complex structure, $\omega_0$ is a symplectic form which is real-analytic in the
$J$-holomorphic charts, and $\omega_0(\cdot,J\cdot)$ is a Riemannian
metric. Recall also that there is a line bundle $L\to M$ and a
Hermitian metric $h_0$ on $L$ such that
${\rm curv}(h_0)=-i\omega_0$. Associated with $(L,h_0)$ is
a Hermitian connection $\nabla$, whose curvature is also
$-i\omega_0$.

Concretely, in holomorphic charts for $L$
(in which holomorphic sections of $L$ are simply holomorphic
$\C$-valued functions), given $v\in L$ one has
$\|v\|_{h_0}=e^{-\frac{\psi_0(x)}{2}}|v(x)|$ and given a section $s$ and $(x,\dot{x})\in
TM$, one has $\nabla s(\dot{x})=\frac{\dd}{\dd \dot{x}}(e^{\frac{-\psi_0}{2}}s)e^{\frac{\psi_0}{2}}$,
where $\psi_0$ is a local Kähler potential: $\partial \overline{\partial}\psi_0=-i\omega_0$. 

Starting with this data, we will construct several manifolds and
bundles. These definitions serve two goals: first, to properly define
the association between integral kernels and operators as used in
Theorem \ref{thr:FIOs}; second, to give a geometric interpretation of
the phase $\Phi$ appearing in Theorem \ref{thr:FIOs}.

\begin{defn}\label{def:Mbar}Define $\overline{M}$ as the Kähler manifold
$(M,-J,-\omega_0)$, with flipped complex and symplectic
structure. Observe that $-\omega_0(\cdot,-J\cdot)=\omega_0(\cdot,J\cdot)$
so that the Riemannian structure on $\overline{M}$ is identical to
that on $M$. The natural line bundle $\overline{L}$ over
$\overline{M}$ is constructed as follows: given a $J$-holomorphic atlas
of $M$ and
(holomorphic) transition charts for $L$, the same atlas is
$-J$-holomorphic for $\overline{M}$ and we set the transition charts
of $\overline{L}$ to be the complex conjugate of that for $L$. In this
way, $\overline{L}$ is indeed a holomorphic bundle over
$\overline{M}$.

There is an antilinear correspondence between $H^0(M,L)$ and
$H^0(\overline{M},\overline{L})$: associate $s$ with $\overline{s}$ in each chart. Using this
correspondence we can provide $\overline{L}$ with a Hermitian metric
$h_{0,\overline{L}}$ and a Hermitian connection. One recovers then
${\rm curv}(\overline{L},h_{0,\overline{L}})=i\omega_0$, so that
$\overline{L}$ is a prequantum bundle over $\overline{M}$.
\end{defn}

\begin{defn}\label{def:boxtimes}
  The line bundle $L\boxtimes \overline{L}$ over $M\times
  \overline{M}$ is the holomorphic line bundle whose fibre over a point $(x,y)\in
  M$ is $L_x\otimes \overline{L_y}$. 
\end{defn}

The product manifold $M\times \overline{M}$ is particularly
interesting for several reasons. First of all, for the line bundle
$L\boxtimes \overline{L}$ over $M\times \overline{M}$, the holomorphic sections will be integral kernels of operators acting on $H^0(M,L)$. Moreover,
the diagonal in $M\times \overline{M}$ is a copy of $M$ which is a
maximally totally real submanifold of $M\times
\overline{M}$. Consequently, real-analytic data on $M$ can be extended
in a unique way into holomorphic objects in small neighbourhoods of the diagonal in
$M\times \overline{M}$. We can therefore define new geometric structures
on neighbourhoods of the diagonal in $M\times \overline{M}$.

\begin{defn}\label{def:Mtilde}Define $\widetilde{M}$ as a small
  neighbourhood of the diagonal in $M\times \overline{M}$, endowed
  with the complex structure $I=(J,-J)$.

We also extend the Hermitian line bundle $(L,h_0)$ and the
connection $\nabla$ into a $\C$-bundle
$(\widetilde{L},\widetilde{h_0})$ and a Hermitian connection $\widetilde{\nabla}$ over
$\widetilde{M}$, by first extending the Kähler potentials $\psi_0$ in charts into
$\widetilde{\psi_0}$ and then setting
$\|v\|_{\widetilde{h_0}}=|e^{-\frac 12\widetilde{\psi_0}(x)}v|$ and
$\widetilde{\nabla}s(\dot{x})=\frac{\dd}{\dd
  \dot{x}}(e^{-\frac 12\widetilde{\psi_0}}s)e^{\frac 12 \widetilde{\psi_0}}$. One sets then, in charts
\[
  \omega_0:=i\,{\rm
  curv}(\widetilde{\nabla})=\frac{i}{2}
\sum_{j,k}(\widetilde{h_{j,k}}(z,\overline{w})\dd z_j\wedge
\overline{\dd
  w_k}+\overline{\widetilde{h_{j,k}}(z,\overline{w})}\overline{\dd
  z_j}\wedge \dd w_k) \qquad \qquad
h_{j,k}=\frac{\partial^2\psi_0}{\partial z_j\partial \overline{z_k}}.
\]

There is a second natural holomorphic structure on $\widetilde{M}$,
denoted by $\widetilde{J}=(J,J)$. Note that $\widetilde{J}$ preserves the
tangent space of the diagonal of $M\times \overline{M}$, where it
coincides with $J$, and $\widetilde{J}$ commutes with $I$. Given a $J$-holomorphic object on $M$, its $I$-holomorphic
extension to $\widetilde{M}$ is also $\widetilde{J}$-holomorphic.
\end{defn}

We insist
that $\widetilde{\omega_0}$ is very different from the already available
symplectic form on $M\times \overline{M}$ (stemming from the symplectic
structure on each factor), which is a real-valued, non holomorphic,
symplectic form. In the same way, $\widetilde{L}$ is not $L\boxtimes \overline{L}$.

The curvature of $\widetilde{\nabla}$ is positive in the following sense.

\begin{prop}\label{prop:complex_positivity}Consider the following
  anti-linear involution on $T^{\C}\widetilde{M}$:
  \[
    \sigma:\sum_{j}(v^1_j\dd z_j+v^2_j\overline{\dd z_j}+v^3_j\dd
    w_j+v^4_j\overline{\dd w_j})\mapsto
    \sum_{j}(\overline{v^1_j}\overline{\dd
      w_j}+\overline{v^2_j}\dd w_j+\overline{v^3_j}\overline{\dd
      z_j}+\overline{v^4_j}\dd z_j).
    \]
  Let $v\in T^{\C}\widetilde{M}$ nonzero and suppose that $Iv=iv$ (so that $v$
  is an $I$-holomorphic tangent vector). Then
  \[
    \widetilde{\omega_0}(v,\sigma\widetilde{J}v)>0.
  \]
\end{prop}
The involution $\sigma$ maps $I$-holomorphic vectors into
$I$-antiholomorphic vectors and also maps $\widetilde{J}$-holomorphic
vectors into $\widetilde{J}$-antiholomorphic vectors. Identifying the
diagonal of $M\times \overline{M}$ with $M$, $\sigma$ coincides with
the natural involution on $T^{\mathbb{C}}M$. It is the only
anti-linear involution with these properties.
\begin{proof}
  In a chart, write
  \[
    v=\sum_{j} v^1_j\dd z_j+\sum_{j}v^2_j\overline{\dd w_j}.
  \]
  Then
  \[
    \widetilde{J}v=i\sum_{j}v^1_j\dd
    z_j-i\sum_{j}v^2_j\overline{\dd w_j},
  \]
  so that
  \[
    \widetilde{\omega_0}(v,\sigma\widetilde{J}v)=\frac 12\sum_{j,k}[\widetilde{h_{j,k}}(z,\overline{w})+\overline{\widetilde{h_{j,k}}(z,\overline{w})}](v^1_j\overline{v^1_k}+v^2_j\overline{v^2_k}).
  \]
  The Hermitian matrix under brackets is a perturbation of
  $2h_{j,k}(z,\overline{z})$, which is positive definite. Therefore it
  is positive definite as well, hence the claim.
\end{proof}

\begin{defn}
The program of holomorphic extension of a real-analytic Kähler
manifold applied in the last two paragraphs can be applied to $M\times
\overline{M}$ as well. We obtain a manifold which we will denote
$\widetilde{M}\times \widetilde{\overline{M}}$.

We will consider two natural
holomorphic structures on $\widetilde{M}\times
\widetilde{\overline{M}}$. The first one is $(I,I)$, denoted also by $I$, and the second is
$(\widetilde{J}, -\widetilde{J})$, denoted also by $\widetilde{J}$. The symplectic form on
$\widetilde{M}\times \widetilde{\overline{M}}$ will be
$(\widetilde{\omega_0},-\widetilde{\omega_0})$. Holomorphic objects on
$M\times \overline{M}$ again enjoy the property that their $I$-holomorphic
extension to $\widetilde{M}\times \widetilde{\overline{M}}$ is also
$\widetilde{J}$-holomorphic.
\end{defn}
Note that the diagonal of $\widetilde{M}\times
\overline{\widetilde{M}}$ is $I$-holomorphic but
$\widetilde{J}$-totally real.

Let us motivate the introduction of $\widetilde{M}\times
\widetilde{\overline{M}}$: with formula
\eqref{eq:gammak_FIO_intro} as our objective, we are trying to describe the geometry
underlying operators of the form $\exp(tkT_k(f))$ when $f\in
C^{\omega}$, $t\in \R$ is small but fixed, and $k\to +\infty$. For every
$k$, these operators have analytic dependence on $t$ and we may as
well consider $\exp(-itkT_k(f))$ for $t\in \R$. Such operators are now
well-known at least modulo $O(k^{-\infty})$: they correspond to
integral kernels of the form \eqref{eq:gammak_FIO_intro}, for
particular sections $\Phi(it)$ which ``correspond to'' (in a sense we will
make precise) Lagrangians of $M\times
\overline{M}$. The kernel of $\exp(-itkT_k(f))$ is
in fact a ``Lagrangian section'' of $L\boxtimes \overline{L}$ as introduced in \cite{charles_quasimodes_2003}.  These Lagrangians are the graphs of the
Hamilton flow of $f$ at time $t$. It is
natural to expect $\Phi(t)$ to ``correspond to'' (in the same sense)
the graph of the Hamilton flow of $f$ at imaginary time $-it$. This
graph, however, does not sit inside $M\times \overline{M}$ anymore; it
will be a $I$-holomorphic Lagrangian of $\widetilde{M}\times
\widetilde{\overline{M}}$.

We now explain how some holomorphic sections of $L\boxtimes
\overline{L}$ are associated with holomorphic Lagrangians of
$\widetilde{M}\times \widetilde{\overline{M}}$. The first holomorphic
section of interest is associated with the diagonal of
$\widetilde{M}\times \widetilde{\overline{M}}$ and will be the phase
of the Bergman kernel.

\begin{prop}\label{prop:Psi_Bargmann_Identity}
  Let $(M,J,\omega_0)$ be a quantizable Kähler manifold and let
  $(L,h_0)$ be a prequantum line bundle over $M$. Let $\Psi$ be the
  unique holomorphic section of $L\boxtimes \overline{L}$ over a
  neighbourhood of the diagonal in $M\times \overline{M}$ such that
  $h_0(\Psi)=1$ on the diagonal.

  Let $\widetilde{\Psi}$ denote the $I$-holomorphic extension of $\Psi$ to
  a neighbourhood of the diagonal of $M\times \overline{M}$ in
  $\widetilde{M}\times \widetilde{\overline{M}}$. Let $\widetilde{\nabla}$ denote
  the natural connection on $\widetilde{L}\boxtimes
  \widetilde{\overline{L}}$.

  Then (up to further restricting the neighbourhood of the diagonal) $\widetilde{\nabla}
  \widetilde{\Psi}$ vanishes on the diagonal of $\widetilde{M}\times
  \widetilde{\overline{M}}$ and nowhere else; moreover $\widetilde{\nabla}
  \widetilde{\Psi}$ is a defining function for its zero set: the
  distance to the diagonal is comparable to $\|\widetilde{\nabla} \widetilde{\Psi}\|$.
\end{prop}

To make sense of the conditions on $\Psi$, note that for every $x\in
  M$, $\Psi(x,x)$ is an element of $L_x^*\otimes \overline{L_x^*}$, on
  which $h_0(x)$ acts as a nonzero linear form. Usually (see for instance \cite{charles_berezin-toeplitz_2003}), one uses $h_0$ to
  identify $\overline{L_x}$ with $L_x'$, and the condition becomes
  $\Psi(x,x)=1$.

\begin{proof}For the moment let us consider matters on $M\times
  \overline{M}$, equipped with its connection $\nabla$. Observe that
since $h_0(\Psi)=1$ on the diagonal, $\nabla \Psi$ vanishes on the
diagonal of $M\times \overline{M}$ (it certainly vanishes along the diagonal directions, and
moreover $\Psi$ is holomorphic so that
$\nabla_{J\dot{x}}\Psi=i\nabla_{\dot{x}}\Psi$ for all $\dot{x}\in
T(M\times \overline{M})$; to conclude, the diagonal is totally real in
$M\times \overline{M}$). Because $i{\rm curv }\nabla$ is the symplectic
form on $M\times \overline{M}$, sets of the form $\{x\in M, \nabla
\Phi(x)=0\}$, for $\Phi$ a general nonvanishing section of $L\boxtimes
\overline{L}$, have to be isotropic for this symplectic
form. Therefore (up to restricting our attention to a smaller open
neighbourhood of the diagonal) $\{\nabla \Psi=0\}$ is
the diagonal of $M\times \overline{M}$ and moreover (using again the
curvature identity and Proposition \ref{prop:complex_positivity}), $\|\nabla \Psi\|$ is comparable to the distance
to the diagonal.

Letting now $\widetilde{\Psi}$ be the holomorphic extension of $\Psi$, defined in a neighbourhood of
the the diagonal of $M\times \overline{M}$ in $\widetilde{M}\times
\widetilde{\overline{M}}$, for the same reasons the equation $\widetilde{\nabla} \widetilde{\Psi}=0$
defines a $I$-holomorphic isotropic submanifold of $\widetilde{M}\times
\widetilde{\overline{M}}$ for the natural symplectic
form; since its restriction to the real set is the Lagrangian ${\rm
  diag}(M)$, we conclude that $\{\widetilde{\nabla} \widetilde{\Psi}=0\}={\rm
  diag}(\widetilde{M})$, and $\widetilde{\nabla} \widetilde{\Psi}$ is still a
defining function for its zero set.
\end{proof}

\begin{prop}\label{prop:well-def_Lagrangians}
  Let $\Phi$ be a holomorphic section of $L\boxtimes \overline{L}$
  over a neighbourhood of the diagonal of $M\times
  \overline{M}$. Suppose that, in a topology
  of real-analytic functions, the function $x\mapsto h_0(\Phi(x,x))$
  is close to 1. Then the holomorphic extension
  $\widetilde{\Phi}$ of $\Phi$ to a section over a neighbourhood of
  the diagonal of $M\times \overline{M}$ in $\widetilde{M}\times
  \widetilde{\overline{M}}$ is such that
  $\{\widetilde{\nabla}\widetilde{\Phi}=0\}$ is a I-holomorphic
  Lagrangian, for which $\widetilde{\nabla}\widetilde{\Phi}$ is a
  defining function.
\end{prop}
\begin{proof}
  Recall that $L\boxtimes \overline{L}$ is a holomorphic line bundle
  over $M\times \overline{M}$, and that the diagonal of $M\times
  \overline{M}$ is a totally real manifold. Since $\Phi$ is a
  holomorphic section, it is determined by its restriction to the
  diagonal. Under our hypotheses, $\Phi$ is close to $\Psi$ when
  restricted to the diagonal, in some suitable real-analytic
  topology. Therefore, in a small neighbourhood of the diagonal,
  $\Phi$ is close to $\Psi$ (as defined in Proposition
  \ref{prop:Psi_Bargmann_Identity}) in a real-analytic topology. In turn, in a
  neighbourhood of the diagonal of $M\times \overline{M}$ in
  $\widetilde{M}\times \widetilde{\overline{M}}$, $\widetilde{\Phi}$
  is close to $\widetilde{\Psi}$ in a real-analytic topology.

  In particular, the equation $\widetilde{\nabla}\widetilde{\Phi}=0$
  still defines a $I$-holomorphic submanifold of half dimension of
  $\widetilde{M}\times \widetilde{\overline{M}}$. Since $i{\rm
    curv}\widetilde{\nabla}=\widetilde{\omega}$, this manifold again has to
  be isotropic and is therefore a Lagrangian, which lies close (in
  real-analytic topology) to the diagonal of $\widetilde{M}\times
  \widetilde{\overline{M}}$.
\end{proof}

Let $\mathcal{L}$ be a $I$-holomorphic Lagrangian of $\widetilde{M}\times
\widetilde{\overline{M}}$, close to the diagonal in real-analytic
topology. Can $\mathcal{L}$ be realised as the vanishing set of
$\widetilde{\nabla}\widetilde{\Phi}$ for some holomorphic section
$\Phi$? The answer relies on the \emph{Bohr-Sommerfeld index}.

\begin{defn}\label{def:BS}~
  \begin{enumerate}
    \item Let $\mathcal{L}$ be a Lagrangian of
  $\widetilde{M}\times \widetilde{\overline{M}}$. Let $\gamma:[0,1]\to
  \mathcal{L}$ be a
  parametrised closed loop. Let $\Gamma(0)\in
  \widetilde{L}\boxtimes \widetilde{\overline{L}}$ nonzero over $\gamma(0)$
  and let $\Gamma$ be the parallel transport of $\Gamma(0)$ along
  $\gamma$ (following the already defined connection $\widetilde{\nabla}$). The
  \emph{Bohr-Sommerfeld index} of $\gamma$ is
  $\Gamma(1)/\Gamma(0)$. Since parallel transport is a linear
  differential equation and $\widetilde{L}\boxtimes
  \widetilde{\overline{L}}$ is a line bundle, the Bohr-Sommerfeld
  index of $\gamma$ does not depend on $\Gamma(0)$.
\item Since $\mathcal{L}$ is Lagrangian, it is $\widetilde{\nabla}$-flat; therefore the Bohr-Sommerfeld index of a loop only
  depends on its topology. Define the \emph{Bohr-Sommerfeld
    class} as the group morphism $\pi_1(\mathcal{L})\to \C^*$.
\item We say that $\mathcal{L}$ is a \emph{Bohr-Sommerfeld Lagrangian}
  when the group morphism above is trivial.
\end{enumerate}
\end{defn}

\begin{prop}\label{prop:Lagr-to-sec}
  Let $\mathcal{L}$ be a $I$-holomorphic Lagrangian of
  $\widetilde{M}\times \widetilde{\overline{M}}$. The equation
  $\widetilde{\nabla}\widetilde{\Phi}=0$ on $\mathcal{L}$, with
  $\widetilde{J}$-holomorphic and $I$-holomorphic unknown $\widetilde{\Phi}$, defines a
  rank one sheaf over $\mathcal{L}$; nonzero local solutions are such
  that $\widetilde{\nabla}\widetilde{\Phi}$ are defining functions for
  $\mathcal{L}$.

  The cohomology of this sheaf is exactly the
  Bohr-Sommerfeld class of $\mathcal{L}$; in particular, a non-zero
  global solution exists if, and only if, $\mathcal{L}$ is a
  Bohr-Sommerfeld Lagrangian.
\end{prop}
In other terms, one can always find $\widetilde{\Phi}$ such that
$\widetilde{\nabla}\widetilde{\Phi}=0$ on $\mathcal{L}$ locally
  near any point of $\mathcal{L}$, and the solution is unique up to a
multiplicative constant. The different solutions can be patched into a
global function $\widetilde{\Phi}$ if and only if $\mathcal{L}$ is Bohr-Sommerfeld.

\begin{proof}
Let us first prove that, locally, there is no
obstruction for the existence of $\widetilde{\Phi}$. Fixing arbitrarily the value of $\widetilde{\Phi}$ at a
point $x_0\in \mathcal{L}$, the equation $\widetilde{\nabla}\widetilde{\Phi}=0$ on
$\mathcal{L}$ determines $\widetilde{\Phi}$ on a neighbourhood of
$x_0$ in $\mathcal{L}$, by parallel transport of $\Phi(x_0)$ along
short paths in $\mathcal{L}$. Since $\mathcal{L}$ is Lagrangian, the
value obtained by parallel transport does not depend on the
short paths but only on their endpoints. Since $\mathcal{L}$ is a
Lagrangian, it is
$\widetilde{J}$-totally real by Proposition
\ref{prop:complex_positivity}, and since $\widetilde{\Phi}$ is $\widetilde{J}$-holomorphic, the
knowledge of $\widetilde{\Phi}$ locally on $\mathcal{L}$ determines
$\widetilde{\Phi}$ on a whole neighbourhood of $x_0$ in $\widetilde{M}\times
\widetilde{\overline{M}}$; using Proposition
\ref{prop:complex_positivity} again, if $\widetilde{\Phi}(x_0)\neq 0$
then
$\widetilde{\nabla}\widetilde{\Phi}$ is constrained by the curvature
identities to be a defining function for $\mathcal{L}$ near $x_0$. In
summary, without any constraint, the data of a $I$-holomorphic Lagrangian $\mathcal{L}$ near
any of its points determines completely $\widetilde{\Phi}$ as in the
claim of Proposition \ref{prop:well-def_Lagrangians} near this
point up to a multiplicative constant.

Let us now study whether these local solutions can be patched
together. Noticing that the restriction of $\widetilde{\Phi}$ to a loop
in $\mathcal{L}$ has to be flat (that is, it solves the parallel
transport equation), by Definition \ref{def:BS} a nonzero solution
$\widetilde{\Phi}$ exists in a neighbourhood of a loop $\gamma \subset
\mathcal{L}$ if, and only if, the Bohr-Sommerfeld index of $\gamma$ is
$1$. Hence, a nonzero solution exists on a neighbourhood of
$\mathcal{L}$ if and only if $\mathcal{L}$ is Bohr-Sommerfeld.
\end{proof}

\subsection{Analytic symbols}
\label{sec:analytic-symbols}

Here we rapidly present the analytic function spaces symbol
spaces which we will rely on; we refer to
\cite{sjostrand_singularites_1982,hitrik_two_2013} for more
in-depth introductions to analytic semiclassical analysis. 

\begin{defn}\label{def:analytic-norms}
  Let $U$ be a real-analytic Riemannian manifold (possibly open) and let $m,r,R>0$.

  The Banach space $H^r_m(U)$ consists of all functions $a$ from $U$ to $\C$
  such that there exists $C>0$ satisfying, for every $j\in \N$,
  \[
    \|\nabla^ja\|_{L^{\infty}(U)}\leq C\frac{r^jj!}{(j+1)^m};
  \]
  the best constant $C$ above is the Banach norm of $a$.

  The Banach space $S^{r,R}_m(U)$ of \emph{formal analytic symbols} consists of all sequences
  $(a_k)_{k\in \N}$ of elements of $H_r^m(U)$ such that there exists
  $C>0$ satisfying, for every $(j,k)\in \N^2$,
  \[
    \|\nabla^ja_k\|_{L^{\infty}(U)}\leq
    C\frac{r^jR^k(j+k)!}{(j+k+1)^m};
  \]
  the best constant $C$ above is the Banach norm of $a$.
\end{defn}
The definition of $S_m^{r,R}$ and the Stirling formula imply that for every $c_1>0,c_2>0$
  small enough (strictly smaller than $\frac{e}{R}$), there exists $c_3>0,C$ such that
  for every $a\in S_m^{r,R}(U)$, for every $\hbar>0$,
  \[
    \left\|
    \sum_{k=c_1\hbar^{-1}}^{c_2\hbar^{-1}}\hbar^ka_k\right\|_{L^{\infty}(U)}\leq
  Ce^{-c_3\hbar^{-1}}\|a\|_{S_m^{r,R}(U)}.
\]
We will call \emph{classical analytic symbol} (or simply analytic
symbol, since we will only encounter classical ones in this text) a function of the form
\[
  U\times (0,1)\ni (x,\hbar)\mapsto \sum_{k=0}^{c\hbar^{-1}}\hbar^ka_k(x)
\]
where $(a_k)_{k\in \N}\in S^{r,R}_m(U)$ and $0<c<\frac{e}{R}$; such a function will
be called a \emph{realisation} of $(a_k)_{k\in \N}$. By the above
inequality, different realisations of the same formal analytic symbol
are exponentially close to each other, and an analytic
symbol is associated with a unique formal analytic symbol.

The point of Definition \ref{def:analytic-norms} is that (classical) analytic symbols are stable by analytic stationary
phase: an expression of the form
\[
  (x;\hbar)\mapsto \int e^{i\frac{\phi(x,y)}{\hbar}}a(x,y;\hbar)\dd y
\]
is, under suitable geometric conditions on $\phi$, if $\phi$ is
real-analytic and $a$ is a classical real-analytic symbol, equal to
$e^{i\frac{\psi(x)}{\hbar}}\hbar^{\frac N2}b(x;\hbar)+O(e^{-c\hbar^{-1}})$ where $\psi$ is
real-analytic, $N$ is an integer,
$b$ is a classical real-analytic symbol, and
$c>0$. We refer to \cite{sjostrand_singularites_1982}, Théorème 2.8
for details.

\subsection{Semiclassical analysis of the Bergman projector}
\label{sec:semicl-analys-bergm}
In this subsection we gather the available results on
Berezin--Toeplitz quantization in analytic regularity that we will
use. We formulate the results in terms of the geometric definitions of
Section \ref{sec:polar-compl-assoc}. We present in particular
\emph{covariant Berezin--Toeplitz operators}, an alternative to
formula \eqref{eq:contra_Toep} which is more suited to our analysis.

\begin{prop}\label{prop:calc_Toeplitz}~\begin{enumerate}
    \item
  Let $(M,J,\omega_0)$ be a compact, real-analytic, polarised Kähler
  manifold; let $(L,h_0)$ be a prequantum line bundle over $M$. Let
  $U$ be a small open neighbourhood of the diagonal in $M\times
  \overline{M}$ and let $\Psi$ be the section of $L\boxtimes
  \overline{L}$ featured in Proposition
  \ref{prop:Psi_Bargmann_Identity} ($\Psi$ is holomorphic on $U$ and
  $h_0(\Psi(x,x))=1$). There exists a
  formal analytic symbol $s\in
  S_{r,R}^m(U)$ with principal symbol
    $s_0=(2\pi)^{-d}$, and constants $c>0$,
  $C>0$ such that for any $u,v\in H^0(M,L^{\otimes k})$,
  \begin{multline}\label{eq:Bergman}
    \left|\langle u,v\rangle_{{\rm Hilb}_k(0)}-k^d\int_U\langle \Psi(x,y)^{\otimes
    k},\overline{u(x)}\otimes v(y)\rangle_{(h_0\otimes
    \overline{h_0})^{\otimes k}}s(x,y;k^{-1})\omega_0^{\wedge d}(\dd x)
    \omega_0^{\wedge d}(\dd y)\right|\\ \leq Ce^{-ck}\|u\|_{{\rm Hilb}_k(0)}\|v\|_{{\rm Hilb}_k(0)}
  \end{multline}
  where $s(x,y;k^{-1})$ denotes any realisation of $s$.
  
  \item Given $r,R,m>0$, $V\subset U$ containing the diagonal of
    $M\times\overline{M}$ and $a\in S^{r,R}_m(V)$, let
  \[
    T_k^{\rm cov}(a):(x,y)\mapsto \mathds{1}_{(x,y)\in V}k^d\Psi^{\otimes
        k}(x,y)s(x,y;k^{-1})a(x,y;k^{-1})
    \]
    where $a(x,y;k^{-1})$ denotes any realisation of $a$.
    Then there exists $m_0>0$, $c>0$, $C>0$, and functions $r_0,R_0$ such that, for
    every $m,r,R$ such that $m\geq m_0,r\geq r_0(m), R\geq R_0(r,m)$,
    for every $a\in S_{r,R}^m(V)$ and $b\in S_{2r,2R}^{m}(V)$ there
    exists $a\#b\in S_{2r,2R}^m(V)$ such that
    \[
      |\|T_k^{\rm cov}(a)T_k^{\rm cov}(b)-T_k^{\rm
        cov}(a\#b)|\|_{H^0\to H^0}\leq
      C\|a\|_{S^{r,R}_m}\|b\|_{S^{2r,2R}_m}e^{-ck},
    \]
    and moreover the bilinear map $(a,b)\to a\#b$ is continuous:
    \begin{align}\label{eq:product_norm_control}
      \|a\#b\|_{S_m^{2r,2R}}\leq C\|a\|_{S^{r,R}_m}\|b\|_{S^{2r,2R}_m}.
    \end{align}
    \item Given $r,R,m>0,c_0>0$, $V\subset U$ containing the
    diagonal of $M\times \overline{M}$ , there exists
    $r',R',m',C,c>0$ and $V'\subset U$ containing the diagonal of $M\times
    \overline{M}$ such that for every $a\in S^{r,R}_m(V)$ with
    principal symbol $a_0$ bounded away from zero on $V$ with
    $|a_0|>c_0$, there exists $b\in S^{r',R'}_{m'}(V')$ with continuous
    dependence on $a$ such that
    \[
      \|T_k^{\rm cov}(a)T_k^{\rm cov}(b)-\Pi_k\|_{H^0\to H^0}\leq
      C\|a\|_{S_m^{r,R}}e^{-ck}.
    \]
\end{enumerate}
  \end{prop}
  \begin{proof}Equation \eqref{eq:Bergman} corresponds to Theorem 5.5
    of \cite{rouby_analytic_2018} and Theorem A of
    \cite{deleporte_toeplitz_2018}.
    
  Equation
    \eqref{eq:product_norm_control} is a consequence of Theorem B in
    \cite{deleporte_toeplitz_2018}; let us prove how to obtain
    \eqref{eq:product_norm_control} from there, since
    the definition of covariant operators here is slightly different from that
    in \cite{deleporte_toeplitz_2018}. Letting $*$ denote the Cauchy
    product of formal symbols, we
    know from \cite{deleporte_toeplitz_2018}, Theorem B, that
    \[
      \|(a\#b)*s\|_{S^{2r,2R}_m}\leq
      C\|a*s\|_{S^{r,R}_m}\|b*s\|_{S^{r,R}_m}
    \]
    for $r,R,m$ as before. Here, without loss of generality, the
    symbol $s$ of the Bergman projector and its Cauchy inverse $s^{-1}$
    lie in a symbol class which injects continuously in
    $S^{r,R}_m$. Thus, by continuity of the Cauchy product
    (\cite{deleporte_toeplitz_2018}, Proposition 3.8),
    we finally obtain the desired result.

     The third part of the proposition is a consequence of the
      second part of 
      Theorem B in \cite{deleporte_toeplitz_2018}, modulo
      multiplication or division by $s$ for the Cauchy product,
      exactly as for \eqref{eq:product_norm_control}.
  \end{proof}
  \begin{rem}[Equivalent analytic norms] There is now a large body of
    literature concerning Berezin--Toeplitz operators in
    real-analytic regularity and their symbolic calculus. The analytic
    norms of Definition \ref{def:analytic-norms}, used in \cite{deleporte_toeplitz_2018},
    lead to lengthy proofs, but to this date, it is the only way to
    obtain, globally, analytic norms which are relatively stable under
    composition: the Banach space of $a\#b$ is the
    same as that of $b$, if $a$ is more regular. This contrasts with,
    e.g., Proposition 4.1 in \cite{charles_analytic_2021}, where there
    is a loss of regularity. We will use this fact later.
  \end{rem}
  
  A consequence of the Bergman projector asymptotics
  in Proposition \ref{prop:calc_Toeplitz} (for which real-analyticity
  is in fact a very strong requirement) is that ${\rm FS}_k$ is
  approximately a left inverse to ${\rm Hilb}_k$.
  \begin{prop}\label{prop:Tian-Zelditch}\cite{tian_set_1990,zelditch_szego_2000}As
    $k\to \infty$, for every $j \in \N$, $\|{\rm FS}_k({\rm Hilb}_k(0))\|_{C^j}=O(k^{-1})$.
  \end{prop}

  \begin{rem}\label{rem:FS_k-Tkcov}~
    \begin{enumerate}
    \item Using the notations of Proposition \ref{prop:calc_Toeplitz},
      one has, for every $x\in M$ and every $k$ large enough,
      \[
        {\rm FS}_k({\rm
          Hilb}_k(0))(x)=\frac{1}{k}\log(k^ds(x,x;k^{-1}))-\frac{d\log(k)}{k}=\frac{1}{k}\log(s(x,x;k^{-1}))
      \]
      which motivates our alternative choice for the definition of
      ${\rm FS}_k$.
    \item 
    The method of proof of \cite{zelditch_szego_2000} allows to
    generalise Proposition \ref{prop:Tian-Zelditch} into the fact that
    \[
      {\rm FS}_k(T_k^{\rm cov}(f))=O_{C^{\infty}}(k^{-1})
    \]
    for every $f$ such that $T_k^{\rm cov}(f)$ is positive and
    self-adjoint, i.e. $f$ is real and positive on the diagonal.
    \end{enumerate}
  \end{rem}
    The map ${\rm FS}_k\circ {\rm Hilb}_k$ is independent of the choice of
    the reference Kähler metric, so that we obtain ${\rm FS}_k({\rm Hilb}_k(\phi))= \phi+O(k^{-1})$
    for every $\phi\in \mathcal{H}$ corresponding to a smooth
    metric.

  We also recall the following subprincipal identities concerning
  Berezin--Toeplitz quantization in the covariant and contravariant case.

  \begin{prop}[\cite{charles_berezin-toeplitz_2003}, page 4 and
    \cite{deleporte_toeplitz_2018}, Proposition 4.11]\label{prop:calc_Toep_subp}
    Let $f$ and $g$ be analytic symbols defined near the diagonal of $M$. Let $b$ be the
    holomorphic extension of $\langle \partial f,
    \overline{\partial}g\rangle_{\Omega^{(0,1)}(M),\Omega^{(1,0)}(M)}$
    to a neighbourhood of the diagonal in $M\times \overline{M}$. Then
    \begin{align}
      T_k^{\rm cov}(f)T_k^{\rm cov}(g)&=T_k^{\rm cov}(fg)+k^{-1}T_k^{\rm
                          cov}(b)+O(k^{-2})\label{eq:Charles_1}\\
    T_k^{\rm cov}(f)&=T_k(f)+k^{-1}T_k(\Delta_{\phi}f)+O(k^{-2})\label{eq:Charles_2}.
    \end{align}
    More generally, for every analytic symbol $f$, there exists
    an analytic symbol $f^{\rm cov}$ such that
    \[
      T_k^{\rm cov}(f^{\rm cov})=T_k(f).
    \]
  \end{prop}
  

\section{Analytic Fourier Integral Operators close to identity}
\label{sec:analyt-four-integr}
By formula \eqref{eq:Bergman}, the Bergman projector $\Pi_k$ has,
asymptotically and near the diagonal, an expression of Wentzel-Kramers-Brillouin type, with
a phase (fixed quantity to the power $k$) multiplied by an analytic
symbol. Let us generalise this expression into a definition.

\begin{defn}\label{def:FIOs}Recall the analytic symbol $s$ of the Bergman kernel from
  Proposition \ref{prop:calc_Toeplitz}.
  
  An analytic Fourier Integral Operator (FIO) close to identity is a
  sequence of sections of $L\boxtimes \overline{L}$ of the form
  \[
    I^{V,\Phi}_{k}(a):(x,y)\mapsto k^d\mathds{1}_{(x,y)\in V}\Phi^{\otimes
      k}(x,y)s(x,y;k^{-1})a(x,y;k^{-1}),
  \]
  where
  \begin{itemize}
  \item $V$ is a neighbourhood of the diagonal in $M\times
    \overline{M}$.
  \item $\Phi$ is a holomorphic section of $L\boxtimes
  \overline{L}$ over $V$.
\item For all $(x,y)\in \partial V$, one has $\|\Phi(x,y)\|_{h_0\otimes
  \overline{h_0}}<1$ (and
therefore by compactness $\|\Phi\|_{h_0\otimes \overline{h_0}}$ is bounded away from $1$ on the
  boundary).
\item $s$ is the symbol of the Bergman kernel of Proposition \ref{prop:calc_Toeplitz}.
\item $a(x,y;k^{-1})$ is a (classical) analytic symbol on $V$, and is holomorphic.
\end{itemize}

By convention,
$I^{V,\Phi}_k(a)$ also denotes the associated operator on
$H^0(M,L^{\otimes k})$ defined as follows: given two holomorphic sections
$u,v$ of $L^{\otimes k}$,
\begin{equation}\label{eq:op_kernel}
  \langle u,I^{V,\Phi}_k(a)v\rangle_{{\rm Hilb}_k(0)}:=\int_V \langle I^{V,\Phi}_k(a),
  \overline{u(x)}\otimes
  v(y)\rangle_{(h_0\otimes \overline{h_0})^{\otimes
      k}}\omega_0^{\wedge d}(\dd x)\omega_0^{\wedge d}(\dd y).
\end{equation}
\end{defn}
Covariant Toeplitz operators are examples of Analytic FIOs with $\Phi=\Psi$.

This section is devoted to the general properties of Analytic Fourier
Integral Operators. As in the usual (self-adjoint) case, they are
associated with Lagrangians, which allow to understand geometrically
and manipulate these operators: composition, inversion, change of
reference Kähler metric, and associated metric under ${\rm FS}_k$. All these
manipulations will be useful in the proof of Theorem \ref{thr:FIOs}.

Let us first prove a rather weak bound on the operator norm of these
operators.
\begin{prop}\label{prop:bound-op-norm}
   Let $I^{V,\Phi}_k(a)$ be an analytic FIO close to identity. Then,
  for the operator norm given by ${\rm Hilb}_k(0)$, the operator norm
  of $I^{V,\Phi}_k(a)$ satisfies
  \[
    \forall \epsilon, \exists C, \forall k, \|I^{V,\Phi}_k(a)\|_{H^0\to H^0}\leq
    C\left(\sup_{x,y}|\Phi(x,y)|_{h_0\otimes \overline{h_0}}\right)^{k(1+\epsilon)},
  \]
\end{prop}
\begin{proof}
  In Formula \eqref{eq:op_kernel}, if both $u$ and $v$ are
  normalised in ${\rm Hilb}_k(0)$ we want to bound the right-hand side
  from above. First of all, since $u$ and $v$ are holomorphic we
  obtain
  \[
    \sup_{x,y}\|u(x)\otimes v(y)\|_{h_0\otimes \overline{h_0}}\leq C_0k^{2d}.
  \]
  The integral on the right-hand-side of
  \eqref{eq:op_kernel} is therefore bounded (in charts) by
  \[
    C_0k^{3d}\sup(|a|)\left(\sup_{x,y}\|\Psi \|_{h_0\otimes \overline{h_0}}\right)^{k}.
  \]
\end{proof}
Definition \ref{def:FIOs} is relative to the particular Kähler
structure $(M,J,\omega_0)$. One can in fact change the Kähler
structure to $\omega_{\phi}$ for $\phi$ close to $0$ and obtain
another analytic Fourier Integral Operator. The relationship between
the two operators is not explicit at this stage, but one can compute
the first two derivatives of the map between the phases.
\begin{prop}\label{prop:manip_kernel_ops}~
 Let $I^{V,\Phi}_k(a)$ be an analytic Fourier Integral
      opeartor close to identity. There exists $c>0$ such that the
      following is true. Let $\phi\in
  \mathcal{H}$ be real-analytic and sufficiently close to $0$ in
  real-analytic topology.

  Then there exists a smaller open neighbourhood $V'$ of the diagonal
  in $M\times \overline{M}$, a holomorphic section $\Phi'$ of $L\boxtimes \overline{L}$
  over $V'$, a holomorphic classical analytic symbol $b$ on $V'$, such that for every two holomorphic
  sections $u,v$ of $L^{\otimes k}$,
  \[
    \langle u, I_k^{V,\Phi}(a)v\rangle_{{\rm Hilb}_k(0)}=\langle
    u,I_k^{V',\Phi'}(b)v\rangle_{{\rm Hilb}_k(\phi)}+O(e^{-ck}\|u\|_{{\rm Hilb}_k(0)}\|v\|_{{\rm Hilb}_k(0)}).
  \]
  If $\phi(t)$ is a differentiable one-parameter family of
  real-analytic Kähler potentials and $\Phi(t)$ is a differentiable
  one-parameter family of real-analytic sections of $L\boxtimes
  \overline{L}$ with $\phi(0)=0$ and $\Phi(0)=\Psi$ (the phase of the
  Bergman kernel), then $\Phi'(0)=\Psi$ and on the diagonal
  \begin{equation}\label{eq:var_phase}
    \left.\frac{\dd \log \Phi'}{\dd t}\right|_{t=0}=\left.\frac{\dd
        \log\Phi}{\dd t}\right|_{t=0}+2\left.\frac{\dd\phi}{\dd t}\right|_{t=0}.
  \end{equation}
  If $\Phi$ and $\phi$ are twice-differentiable, then so is $\Phi'$,
  and if in addition $\frac{\dd \log\Phi}{\dd t}|_{t=0}$ is real, then
  on the diagonal
  \begin{equation}\label{eq:var2_phase}
     \left.\frac{\dd^2 \log \Phi'}{\dd t^2}\right|_{t=0}=\left.\frac{\dd^2
        \log\Phi}{\dd t^2}\right|_{t=0}+2\left.\frac{\dd^2
        \phi}{\dd t^2}\right|_{t=0}+4{\rm Re}\langle
    \partial \tfrac{\dd \log \Phi}{\dd
      t}|_{t=0},\overline{\partial}\tfrac{\dd \phi}{\dd
      t}|_{t=0}\rangle+2\langle \partial\tfrac{\dd \phi}{\dd t}|_{t=0},\overline{\partial}\tfrac{\dd \phi}{\dd
      t}|_{t=0}\rangle;
  \end{equation}
  all scalar products are taken with respect to the Kähler structure $(M,J,\omega_0)$.
\end{prop}
\begin{proof}
  One can of course write, for $u,v\in H^0$,
    \[
      \langle u,I^{V,\Phi}_k(a)v\rangle_{{\rm Hilb}_k(0)}=\int \langle
      (\Phi(x,y)e^{\phi(x)+\phi(y)})^{\otimes k},
      u(x)\otimes \overline{v(y)}\rangle_{(h_{\phi}\otimes
        \overline{h_{\phi}})^{\otimes k}}a(x,y;k^{-1})V(x)V(y)
      \omega_{\phi}^{\wedge d}(\dd x)\omega_{\phi}^{\wedge d}(\dd y),
    \]
    where $V(x)$ is the ratio between the volume forms
    $s_0\omega_0^{\wedge d}$ and $s_{\phi}\omega_\phi^{\wedge d}$. This
    expression, however, involves a non-holomorphic section of
    $L\boxtimes \overline{L}$, since neither $\phi$ nor $V$ are
    holomorphic. Nonetheless, by definition of the Bergman kernel
    $\Pi_k^{\phi}$, and introducing
    \begin{equation}\label{eq:K_as_int}
      K:(x,y)\mapsto \int_{(z,w)\in V}\langle \Pi_k^{\phi}(x,z)\otimes
      \Pi_k^{\phi}(y,w),\Phi(z,w)^{\otimes k}e^{k\phi(z)+k\phi(w)}\rangle_{(h_\phi\otimes
        \overline{h_\phi})^{\otimes k}}a(z,w)V(z)V(w)\dd z \dd
      w,
    \end{equation}
    $K$ is a holomorphic section of $L\boxtimes \overline{L}$, and by
    definition of $\Pi_k^{\phi}$,
    \[
      \langle u,I^{V,\Phi}_k(a)v\rangle_{{\rm Hilb}_k(0)}=\int \langle
      K(x,y),
      u(x)\otimes \overline{v(y)}\rangle_{(h_{\phi}\otimes
        \overline{h_{\phi}})^{\otimes k}}
      \omega_{\phi}^{\wedge d}(\dd x)\omega_{\phi}^{\wedge d}(\dd y).
    \]
    It remains to show that $K$ is exponentially small away from
    the diagonal and of FIO form near the diagonal. The spirit of the
    proof is that
    in charts,
    one can compute $K$ from the formula \eqref{eq:K_as_int}
    by stationary phase. Let us do so explicitly enough to recover
    formulas \eqref{eq:var_phase} and \eqref{eq:var2_phase}. We will write down
    \eqref{eq:K_as_int} in a Hermitian chart for $h_{\phi}$. First of all,
    we write
    \[
      \Phi(z,w)=\Psi(z,w)\exp(f(z,w)),
    \]
    where $f$ is a holomorphic function on $V$ close to $0$.
    We know that, in a Hermitian chart for $h_0$, given by a Kähler
    potential $\psi_0$, the phase $\Psi$ of the Bergman kernel $\Pi_k$
    reads
    \[
      \Psi(z,w)=\exp(-\tfrac{\psi_0(z)}{2}+\widetilde{\psi_0}(z,w)-\tfrac{\psi_0(w)}{2}).
    \]
    Here as usual $\widetilde{\psi_0}$ denotes the holomorphic
    extension on $V$ of $\psi_0$.

    The kernel of $\Pi^{\phi}$ has a phase of the same form, in the
    Hermitian charts for $h_{\phi}$, where
    $\psi_0$ is now replaced by $\psi_{\phi}=\psi_0+\phi$.
    
    The transition function from the Hermitian chart for $h_0$ to the
    Hermitian chart for $h_{\phi}$ is given by multiplication by
    $e^{-\frac{\phi}{2}}$. Consequently, in the Hermitian chart for
    $h_{\phi}$, the section $(z,w)\mapsto
    \Phi(z,w)e^{\phi(z)+\phi(w)}$ is written as
    \[
      \exp(-\tfrac{\psi_{\phi}(z)}{2}+\widetilde{\psi_{\phi}}(z,w)-\tfrac{\psi_{\phi}(w)}{2})\exp(f(z,w)-\widetilde{\phi}(z,w)+\phi(z)+\phi(w)).
    \]
    We are now ready to compute the phase $\Phi'$ in the case where
    both $f$ and $\phi$ are infinitesimally close to $0$. If $x$ and
    $y$ are close enough, we obtain, computing every section in the
    Hermitian chart for $h_{\phi}$,
    \[
      K(x,y)=\int_{W(x,y)}e^{kF(x,z,w,y)}A(x,z,w,y)\dd z \dd w
    \]
    where $W(x,y)$ is a neighbourhood of $\{x\}$ and where the
    holomorphic extension of $F$ reads
    \begin{multline}\label{eq:phase_K}
      \widetilde{F}(x,\overline{x},z,\overline{z},w,\overline{w},y,\overline{y})=
-\tfrac{\widetilde\psi_{\phi}(x,\overline{x})}{2}+\widetilde\psi_{\phi}(x,\overline{z})-\widetilde\psi_{\phi}(z,\overline{z})+\widetilde\psi_{\phi}(z,\overline{w})-\widetilde\psi_\phi(w,\overline{w})+\widetilde\psi_\phi(w,\overline{y})-\tfrac{\widetilde\psi_\phi(y,\overline{y})}{2}\\
    +\widetilde{f}(z,\overline{w})-\widetilde\phi(z,\overline{w})+\widetilde\phi(z,\overline{z})+\widetilde\phi(w,\overline{w}).
  \end{multline}
  Suppose $(x,\overline{x})=(y,\overline{y})$ lies on the real locus
  of $\widetilde{M}$.
  Then, the second line is close to $0$ in real-analytic topology, and
  the first line is a positive phase function of $(z,w)$ with a unique
  critical point. This critical
  point is $(z,\overline{z},w,\overline{w})=(x,\overline{x},x,\overline{x})$, where the first line is equal to $0$.

  For $f$ and $\phi$ small enough, and for
  $(x,\overline{x},y,\overline{y})$ close to the diagonal of $M$, the function \eqref{eq:phase_K} is
  still a positive phase function of $(z,\overline{z},w,\overline{w})$ with a unique critical
  point. Therefore we can apply analytic stationary phase
  (\cite{sjostrand_singularites_1982}, Théorème 2.8) and obtain a
  formula of the form
  \[
    K(x,y)=\Phi'(x,y)^{\otimes k}b(x,y;k^{-1})+O(e^{-ck})
  \]
  for $(x,y)$ close to the diagonal, where $\Phi'$ is a section of
  $L\boxtimes \overline{L}$, $b$ is an analytic symbol, and $c>0$. The
  section $\Phi'$ will be close (in real-analytic topology) to $\Psi$,
  and the neighbourhood of the diagonal of $M\times \overline{M}$ on
  which one can perform analytic stationary phase can be conisdered
  fixed for $\phi,f$ small, so that if $\phi,f$ are even smaller,
  $|\Phi'|_{h_{\phi}\otimes \overline{h_{\phi}}}<1$ on the boundary of
  this domain. If $(x,y)$ lies away from the diagonal of $M\times M$,
  then in \eqref{eq:phase_K} either $(x,z)$, $(z,w)$, or $(w,y)$ are
  away from the diagonal, and since $\Phi$ close to $\Psi$, the
  integrand is then $O(e^{-ck})$.

  To prove formula \eqref{eq:var_phase} we look at \eqref{eq:phase_K}
  again: if $\phi$ and $f$ are infinitesimally close to $0$ then the
  critical point $(z_c(x),\overline{z_c(x)},w_c(x),\overline{w_c(x)})$ is close to $(x,\overline{x},x,\overline{x})$ so that the
  value of $F$ at the critical point is close to
  $f(x)+\phi(x)$. This identity holds in the Hermitian chart for
    $h_{\phi}$, and therefore in the (fixed) Hermitian chart for
  $h_0$ the value of $F$ at the critical point is close to
  $f(x)+2\phi(x)$. Therefore \eqref{eq:var_phase} holds.

  To prove \eqref{eq:var2_phase} we need to study the term of order $2$ in $\phi$ and
  $f$ of the value at the critical point of $F$ . For this we
  first write down the equations for the critical point as
  \begin{align*}
    0&=-\partial\widetilde\psi_0(z,\overline{z})+\partial\widetilde\psi_0(z,\overline{w})+\partial
    f(z,\overline{w})\\
    0&=\overline\partial\widetilde\psi_0(x,\overline{z})+\overline\partial\widetilde\phi(x,\overline{z})-\overline{\partial}\widetilde\psi_0(z,\overline{z})\\
    0&=\partial\widetilde{\psi}_0(w,\overline{x})+\partial\widetilde\phi(w,\overline{x})-\partial\widetilde\psi_0(w,\overline{w})\\
    0&=-\overline\partial\widetilde\psi_0(w,\overline{w})+\overline\partial\widetilde\psi_0(z,\overline{w})+\overline\partial f(z,\overline{w}),
  \end{align*}
  these equations being obtained by differentiating \eqref{eq:phase_K}
  with respect to $z,\overline{z},w,\overline{w}$ respectively.

  To first order in $f$ and $\phi$, the critical point is then given by
  \begin{align*}
    z^*&=x+A_0(x)^{-1}\overline\partial\phi(x)+O(|f|^2+|\phi|^2)\\
    \overline{z}^*&=\overline{x}+A_0(x)^{-1}\partial(f+\phi)(x)+O(|f|^2+|\phi|^2)\\
    w^*&=x+A_0(x)^{-1}\overline\partial(f+\phi)(x)+O(|f|^2+|\phi|^2)\\
    \overline{w}^*&=\overline{x}+A_0(x)^{-1}\partial\phi(x)+O(|f|^2+|\phi|^2),
  \end{align*}
  where $A_0(x)$ is the positive definite matrix $\partial
  \overline{\partial}\psi_0(x,x)$.

  If $f$ is real, the computations for the second order contribution
  become quite symmetrical, and we find
  \[
    \widetilde{F}(x,\overline{x},z^*,\overline{z}^*,w^*,\overline{w}^*,x,\overline{x})=f(x)+\phi(x)
    +4{\rm Re}(\partial f \cdot
    A_0(x)^{-1}\overline{\partial}\phi)+2\partial \phi \cdot
    A_0(x)^{-1}\overline\partial\phi+O(|f|^3+|\phi|^3).
  \]
 The products involved are exactly the scalar products between
  (anti)-holomorphic gradients for the Kähler structure
  $\omega_0$. Remembering to add $\phi(x)$ to $F$ to obtain an
  expression in the Hermitian chart for $h_0$, the proof is complete. 
  \end{proof}

If $\Phi$ is close to $\Psi$ in a real-analytic topology, then by the
results of the previous section, to $\Phi$ one can associate a
holomorphic Lagrangian of $\widetilde{M}\times
\widetilde{\overline{M}}$. This Lagrangian encodes several properties
of $I^{V,\Phi}_k(a)$, and allows for instance to understand the
composition law on analytic Fourier Integral operators.

\begin{prop}\label{prop:compo_FIO}
  Let $V$ be a neighborhood of the diagonal in
  $\widetilde{M}$. Let $\Psi$ denote the phase of the Bergman kernel
  in Proposition \ref{prop:Tian-Zelditch}. For every $r,R,m,\epsilon$, there exists $r',R',m',\delta>0,c>0,C$
  and a neighborhood $V'$ of the diagonal in $\widetilde{M}$
  such that the following is true. For every holomorphic sections
  $\Phi_1$ and $\Phi_2$ of $L\boxtimes \overline{L}$ over $V$ such that
  such that
  \[
    \|\Phi_1/\Psi-1\|_{H^r_m(V)}< \delta\qquad \qquad
    \|\Phi_2/\Psi-1\|_{H^{r}_m(V)}< \delta,
  \]
  there exists a holomorphic section $\Phi$ over $V'$ such that
  \[
    \|\Phi/\Psi-1\|_{H^{r'}_{m'}(V')}<\epsilon
  \]
  (and the map $(\Phi_1,\Phi_2)\mapsto
  \Phi$ is continuous), and for every $a_1,a_2\in
  S^{r,R}_m(V)$ there exists
  $a\in S^{r',R'}_{m'}(V')$ (depending continuously on $a_1,a_2,\Phi_1,\Phi_2$), satisfying
  \[
    \|a\|_{S^{r',R'}_{m'}(V')}\leq
    C\|a_1\|_{S^{r,R}_m(V)}\|a_2\|_{S^{r,R}_m(V)}
  \]
  and such that
  \[
    \left\|I^{V,\Phi_1}_{k}(a_1)\circ
    I^{V,\Phi_2}_k(a_2)-I^{V',\Phi}_k(a)\right\|_{L^2\to
    L^2}\leq Ce^{-ck}.
\]
Moreover, denoting
\[
  \mathcal{L}_1=\{\widetilde{\nabla}\widetilde{\Phi}_1=0\} \qquad
  \qquad \mathcal{L}_2=\{\widetilde{\nabla}\widetilde{\Phi}_2=0\}
  \qquad \qquad \mathcal{L}=\{\widetilde{\nabla}\widetilde{\Phi}=0\},
\]
one has
\[
  \mathcal{L}=\mathcal{L}_1\circ \mathcal{L}_2:=\{(x,\overline{z})\in
  \widetilde{M}\times \overline{\widetilde{M}}, \exists y\in
  \widetilde{M}, (x,\overline{y})\in \mathcal{L}_1,(y,\overline{z})\in
  \mathcal{L}_2\}.
  \]
\end{prop}
  \begin{proof}    
    We will use the analytic stationary phase theorem
    (\cite{sjostrand_singularites_1982}, Théorème 2.8) to study, given
    $(x,z)\in M$, the
    integral
    \begin{multline*}
      I^{V,\Phi_1}_{k}(a_1)\circ
    I^{V,\Phi_2}_k(a_2):(x,\overline{z})\mapsto \\ k^{2d}\int_{W(x,\overline{z})}
    h_0(\Phi_1(x,\overline{y})\otimes
    \Phi_2(y,\overline{z}))^ks(x,\overline{y};k^{-1})s(y,\overline{z};k^{-1})a_1(x,\overline{y};k^{-1})a_2(y,\overline{z};k^{-1})\dd
    v(y).
  \end{multline*}
  Here and in the rest of the proof
  \[
    W(x,\overline{z})=\{y\in M, (x,\overline{y})\in V,(y,\overline{z})\in V\}
  \]
  and $v$ is a volume form.
  
  Let us first study the case where $\Phi_1=\Phi_2=\Psi$ and $x=z$. The real-analytic function
  \[
    y\mapsto h_0(\Psi(x,\overline{y})\otimes
    \Psi(y,\overline{x}))
  \]
  is non-zero; it is equal to $1$ at $y=x$ where its derivative
  vanishes. Its modulus is strictly less than $1$ elsewhere, and as $y$
  tends to $x$, 
  \[
    1-|h_0(\Psi(x,y)\otimes \Psi(y,x))|\sim \dist(x,y)^2.
  \]
  Therefore the conditions to apply analytic stationary phase are met:
  without any change of contour, we have an integral of the form
  \[
    \int_Ue^{k\phi(y)}u(y;k^{-1})\dd y
  \]
  where $U$ is a relatively compact open set, $u$ is a classical analytic symbol on $U$,
  $\phi$ is real-analytic on $U$, ${\rm Re}(\phi)\geq 0$ everywhere
  with equality only at $y=y_0$ which is a critical point of $\phi$
  where $\phi(y_0)=0$.
  and a non-degenerate maximum point of ${\rm Re}(\phi)$; all of this
  data has real-analytic dependence on a parameter $x$.

  Modulo contour deformations and computations of the value of the
  phase at the critical point, all the properties above are stable
  under small analytic deformation of the phase $\phi$, and in
  particular we may deform $(\Psi,\Psi)$ into closeby analytic
  sections $(\Phi_1,\Phi_2)$ and move $(x,z)$ to a small neighbourhood
  of the diagonal. Hence, given general $\Phi_1,\Phi_2,x,z$, there exists
  a contour in $\widetilde{M}$ homotopic to $W(x,z)$ (with boundary
  fixed), on which the function
  \[
    y\mapsto h_0(\Phi_1(x,\overline{y})\otimes \Phi_2(y,\overline{z}))
  \]
  is such that its modulus reaches a unique maximum at a point $y_0$ in a
  non-degenerate way; moreover $y_0$ is a critical point for the
  function.
  
  Applying this contour deformation and
  \cite{sjostrand_singularites_1982}, Théorème 2.8, we obtain, if $(x,z)$ is close
  to the diagonal,
  \[
    I_k^{V,\Phi_1}(a_1)\circ
    I_k^{V,\Phi_2}(a_2)(x,\overline{z})=k^d\Phi(x,\overline{z})^{\otimes k}b(x,\overline{z};k^{-1}).
  \]
  Here, $b$ is the realisation of a classical analytic symbol, which
  continuously depends on $a,\Phi_1,\Phi_2$ for some analytic symbol topologies.
  
  Since the starting integral is $(J,-J)$-holomorphic with respect
  to $(x,\overline{z})$, so are $\Phi$ and $b$; moreover $\Phi(x,\overline{x})$ is close to
  $\Psi(x,\overline{x})=1$ therefore $\Phi$ is close to $\Psi$ everywhere.
  
  If $x$ and $z$ are too far away from the diagonal, then $I^{V,\Phi_1}_{k}(a_1)\circ
    I^{V,\Phi_2}_k(a_2)(x,\overline{z})$ is small anyway; indeed we've already seen
    that, in a fixed size neighbourhood of the diagonal,
    \[
      |h_0(\Psi(x,\overline{y})\otimes
      \Psi(y,\overline{z}))|=|\Psi(x,\overline{y})|_{h_0}|\Psi(y,\overline{z}))|_{h_0}\leq 1-\dist(x,z)^2/2
    \]
    and therefore, under the hypotheses of the claim, if
    $\dist(x,z)^2>4\delta$, then
    \[
      |\Psi(x,\overline{y})|_{h_0}|\Psi(y,\overline{z}))|_{h_0}\leq 1-\dist(x,z)^2/4.
    \]
    Therefore (up to reducing the allowed value of $\delta$) the
    manifold $M\times \overline{M}$ decomposes into a neighbourhood
    $V'$ of the diagonal, where one can apply the stationary phase
    argument above, and its complement set, where $I_k^{V,\Phi_1}(a_1)\circ
    I_k^{V,\Phi_2}(a_2)(x,\overline{z})$ is exponentially small.

    It remains to prove the stated identity between the Lagrangians
    associated with $\Phi_1,\Phi_2,\Phi$.

    Let us first prove that
    \[
      \mathcal{L}_1\circ \mathcal{L}_2\subset \mathcal{L}.
    \]
    To this end, let $(\tilde{x},\tilde{y},\tilde{z})\in \widetilde{M}$, close to each other, and such
    that
    \[
      \widetilde{\nabla}\widetilde{\Phi}_1(\tilde{x},\tilde{y})=0 \qquad
      \qquad \widetilde{\nabla}\widetilde{\Phi}_2(\tilde{y},\tilde{z})=0.
    \]
    The holomorphic extension of \[
      M^3\ni (x,y,z)\mapsto h_0(\Phi_1(x,\overline{y} )\otimes
      \Phi_2(y,\overline{z}))
    \]
    is
    \[
      \widetilde{M}^3\ni (x,\overline{x},y,\overline{y},z,\overline{z})\mapsto
      \widetilde{h_0}(\widetilde{\Phi}_1(x,\overline{y})\otimes
      \widetilde{\Phi}_2(y,\overline{z})).
    \]
    Therefore (because $\widetilde{\nabla}$ is Hermitian for $\widetilde{h_0}$) at the
    considered point one has
    \begin{align*}
      \widetilde{\nabla}_{\tilde{x}}[\widetilde{h_0}(\widetilde{\Phi}_1(x,\overline{y})\otimes
      \widetilde{\Phi}_2(y,\overline{z}))]&=0\\
      \dd_{\tilde{y}}[\widetilde{h_0}(\widetilde{\Phi}_1(x,\overline{y})\otimes
      \widetilde{\Phi}_2(y,\overline{z}))]&=0\\
      \widetilde{\nabla}_{\tilde{z}}[\widetilde{h_0}(\widetilde{\Phi}_1(x,\overline{y})\otimes
      \widetilde{\Phi}_2(y,\overline{z}))]&=0.
    \end{align*}
    In particular, $\widetilde{\Phi}(x,\overline{z})=\widetilde{h_0}(\widetilde{\Phi}_1(x,\overline{y})\otimes
    \widetilde{\Phi}_2(y,\overline{z}))$ (because of the second
    equation), and then, at this point,
    \[
      \widetilde{\nabla}\Phi(x,\overline{z})=0.
    \]
    We now argue that $\mathcal{L}_1\circ
    \mathcal{L}_2=\mathcal{L}$. Because $\Phi_1,\Phi_2,\Phi$ are close
    to $\Psi$, the manifolds $\mathcal{L}_1,\mathcal{L}_2,\mathcal{L}$
    are close to the diagonal. In particular, in $(\widetilde{M}\times
    \widetilde{\overline{M}})^2$, the manifolds $\mathcal{L}_1\times
    \mathcal{L}_2$ and $\{(x,\overline{y},y,\overline{z}),(x,y,z)\in
    \widetilde{M}^3\}$ are transverse. Therefore $\mathcal{L}_1\circ
    \mathcal{L}_2$ and $\mathcal{L}$ have the same dimension, and both
    are close (in analytic regularity) to the diagonal. Therefore the
    inclusion $\mathcal{L}_1\circ \mathcal{L}_2\subset \mathcal{L}$ is
    an equality.
  \end{proof}

One can not only multiply analytic Fourier Integral Operators but also
invert them, as long as their principal symbols are bounded away from zero.

  \begin{prop}\label{prop:inv_FIO}
    Let $V$ be a neighborhood of the diagonal in
  $\widetilde{M}$. Let $\Psi$ denote the phase of the Bergman kernel
  in Proposition \ref{prop:Tian-Zelditch}. For every $r,R,m,c_0,\epsilon>0$, there exists $\delta>0,c>0,C,r',R',m'$
  and a neighborhood $V'$ of the diagonal in $\widetilde{M}$
  such that the following is true.

  For every holomorphic section $\Phi_1$ of $L\boxtimes \overline{L}$
  over $V$ such that
  \[
    \|\Phi_1/\Psi-1\|_{H^r_m(V)}< \delta
  \]
  there exists a holomorphic section $\Phi_2$ of $L\boxtimes
  \overline{L}$ over $V'$ such that
  \[
    \|\Phi_2/\Psi-1\|_{H^{r'}_{m'}(V')}<\epsilon
  \]
  (and in this topology the dependence on $\Phi_1$ is continuous)
  and for every $a_1\in S^{r,R}_m(V)$ holomorphic whose principal
  symbol $a_{1,0}$ satisfies $|a_{1,0}|>c_0$, 
  there exists $a_2\in S^{r',R'}_{m'}$
  depending continuously on $a_1$ and $\Phi_1$, such that
  \[
\left\|
  I^{V,\Phi_1}_k(a_1)I^{V',\Phi_2}_k(a_2)-\Pi_k\right\|_{L^2\to
  L^2}\leq Ce^{-ck}.
    \]
\end{prop}
\begin{proof}
  Let $\mathcal{L}$ be the canonical relation of $\Phi_1$. Let
  \[
    \mathcal{L}^{-1}=\{(x,\overline{y})\in \widetilde{M}\times
    \widetilde{\overline{M}}, (y,\overline{x})\in \mathcal{L}.
  \]
  Then $\mathcal{L}\circ\mathcal{L}^{-1}$ is the diagonal of
  $\widetilde{M}\times \widetilde{\overline{M}}$ (indeed the diagonal
  is clearly included in $\mathcal{L}\circ \mathcal{L}^{-1}$, and
  since $\mathcal{L}$ and $\mathcal{L}^{-1}$ are close to the
  diagonal, they are both graphs, and then so is $\mathcal{L}\circ
  \mathcal{L}^{-1}$).

  Since $\mathcal{L}$ is Bohr-Sommerfeld and the involution
  $(x,\overline{y})\mapsto (y,\overline{x})$ preserves
  $\widetilde{h_0}$, then $\mathcal{L}^{-1}$ is also Bohr-Sommerfeld,
  so that it corresponds to a holomorphic section $\Phi_2$ over a
  neighbourhood $V_1$ of the
  diagonal of $M\times \overline{M}$, by Proposition \ref{prop:Lagr-to-sec}. In some
  analytic norms, $\Phi_2$ depends continuously on $\mathcal{L}^{-1}$,
  which depends continuously on $\mathcal{L}$, which depends
  continuously on $\Phi_1$.

  By Proposition
  \ref{prop:compo_FIO} there exists an analytic symbol $b$ and a
  neighbourhood $V_2$ of the diagonal such that
  \[
    I^{V,\Phi_1}_k(a_1)\circ I^{V_1,\Phi_2}_k(1)=I^{V_2,\Psi}_k(b).
  \]
  We now recall from Proposition \ref{prop:calc_Toeplitz} that
  $I^{V_2,\Psi}_k(b)$ is of the form $T_k^{\rm cov}(r)$ where $r$ is
  an analytic symbol. Moreover, an
  examination of the proof of Proposition \ref{prop:compo_FIO} shows
  that $r$ is obtained from $a_1$ by stationary phase, and therefore
  its principal symbol $r_0$ is bounded away from $0$ if the principal
  symbol of $a_1$ is bounded away from $0$.

  From there, one can apply Proposition \ref{prop:calc_Toeplitz} again
  to obtain that there exists an analytic symbol $q$ such that
  $T_k^{\rm cov}(r)\circ T_k^{\rm cov}(q)=\Pi_k$ up to an
  exponentially small error.

  Now, we apply Proposition \ref{prop:compo_FIO} one last time to
  \[
    I_k^{V_1,\Phi_2}(1)\circ T_k^{\rm cov}(q)
  \]
  which is an analytic Fourier Integral Operator with phase $\Phi_2$,
  since $\mathcal{L}^{-1}\circ {\rm diag}=\mathcal{L}^{-1}$.
\end{proof}
The notion of Lagrangian associated with a phase also allows us to find
the Fubini-Study metric associated with the phase $\Phi$ of a Fourier
integral operator. The Fubini-Study metric associated with
$\Phi$ corresponds to the Kähler potential $\phi$ such that, in the
setting of
Proposition \ref{prop:manip_kernel_ops}, $\Phi'$ is the phase of the
Bergman projector associated with $\phi$. Solving this equation
directly is probably possible using Nash-Moser type arguments
(especially since we understand the differential by Proposition \ref{prop:manip_kernel_ops}), but it
turns out that $\phi$ has a rather direct geometric interpretation
from the Lagrangian $\Lambda$ of $\Phi$. 

\begin{prop}\label{prop:FSk_of_FIO}
  Let $V$ be a neighbourhood of the diagonal in $M\times
  \overline{M}$. Let $\Phi$ be a section of $L\boxtimes \overline{L}$
  over $V$ which is close to the section $\Psi$ of the Bergman kernel
  and which is \emph{self-adjoint}, in the sense that for every
  $(x,y)\in M^2$,
  \[
    \overline{\Phi(x,\overline{y})}=\Phi(y,\overline{x}).
  \]
  Then the Lagrangian
  $\Lambda=\{\widetilde{\nabla}\widetilde{\Phi}=0\}$ is close to
  $\{(x,y,x,y),(x,y)\in \widetilde{M}\}$ and invariant under the
  symmetry $(x,z,w,y)\mapsto (y,w,z,x)$. In
  particular, the intersection of $\Lambda$ with the ``1=4 diagonal''
  $\{y=x\}$ is of the form $\{(x,\kappa(x),\kappa(x),x),x\in M\}$, for some
  $J$-holomorphic diffeomorphism $\kappa$.

  Let $\phi$ be a Kähler potential corresponding to the initial
  Kähler structure pulled back by $\kappa$ (unique up to an additive constant). Then, with the notations of
  Proposition \ref{prop:manip_kernel_ops}, $\Phi'$ is the phase of the
  Bergman projector associated with $\phi+C_0$ for some $C_0\in \R$. In particular, for every
  analytic symbol $a$ holomorphic on $V$ such that
  $a(x,x;k^{-1})\in \R$ and $a_0(x,x)>0$ for every $x\in M$, the operator
  $I_k^{V,\Phi}(a)$ belongs to $S^{++}$ and
  \[{\rm FS}_k(I_k^{V,\Phi}(a))=\phi+C_0+O_{C^{\infty}}(k^{-1}).\]
\end{prop}
\begin{proof}
  Recall the notations of Proposition \ref{prop:manip_kernel_ops} but
  switch the roles played by $\omega_0$ and $\omega_{\phi}$, so as to
  obtain, for every holomorphic section $\Phi_1$ close to $\Psi_0$, a
  holomorphic section $\Phi_2$ close to $\Psi_0$ such that \[\langle u,
  I_k^{V,\Phi_2}(a)v\rangle_{{\rm Hilb}_k(0)}\approx \langle
  u,I_k^{V,\Phi_1}(b)v\rangle_{{\rm Hilb}_k(\phi)}.\]

  Suppose now that $\Phi_1$ is the phase of the Bergman kernel
  for $\phi$. Then the section $\Phi_2$ is self-adjoint; moreover,
  in a Hermitian chart associated with a local Kähler potential
  $\psi_0$, its holomorphic extension to $\widetilde{M}$ reads as
  \[
    \Phi_2(x,\overline{x},y,\overline{y})=\exp\left[-\tfrac{\psi_0(x,\overline{x})}{2}+\psi_0(x,\overline{z})-\psi_0(z,\overline{z})-\phi(z,\overline{z})+\psi_0(z,\overline{y})-\tfrac{\psi_0(y,\overline{y})}{2}\right],
\]
where the point $(z,\overline{z})\in \widetilde{M}$ is determined by
$(x,\overline{y})$ as the unique critical point of the expression
above with respect to $(z,\overline{z})$. Consequently, in this chart,
\begin{align*}
  \nabla_x\log\Phi_2&=\partial \psi_0(x,\overline{z})-\partial
                      \psi_0(x,\overline{x})=\partial
                      \psi_{\phi}(z,\overline{z})- \partial \psi_0(x,\overline{x})\\
  \overline{\nabla}_y\log\Phi_2&=\overline{\partial}\psi_0(z,\overline{y})-\overline\partial\psi_0(y,\overline{y})=\overline\partial\psi_{\phi}(x,\overline{z})-\overline\partial\psi_0(y,\overline{y}).
\end{align*}
The first identity on each line yields
$\Lambda_{\Phi_2}=\{(x,\overline{z},z,\overline{y})\}$, and the second
yields the claimed link between $\Lambda_{\Phi_2}$ and the pulled-back
metric.

To conclude, given $\Phi$ close to $\Psi$ and self-adjoint, there
exists a Kähler potential $\psi$ such that, in the notations of
Proposition \ref{prop:manip_kernel_ops}, the Lagrangian of $\Phi'$ is
the identity. This means that $\Phi'$ coincides with the phase
$\Psi_{\phi}$ of the appropriate Bergman kernel up to a multiplicative
factor, so that
\begin{equation}\label{eq:T_kcov_up_to_mult}
  \langle u,I^{V,\Phi'}_k(b)v\rangle_{{\rm
      Hilb}_k(\phi)}=e^{kC_0}\langle u, T_k^{{\rm
      cov},\phi}(b)v\rangle_{{\rm
      Hilb}_k(\phi)}.
\end{equation}
The sesquilinear form associated with $I_k^{V,\Phi}(a)$ is symmetric if and only if $a$ is
real-valued on the diagonal (that is, if and only if
$\overline{a(x,y)}=a(y,x)$). In turn, this holds if and only if $b$ is
real-valued on the diagonal. By the above considerations and analytic
stationary phase, the principal symbols of $a$ and $b$ are related by
a formula of the form $b_0(z,z)=v(z)a_0(x,x)$, where $z$ is the critical point
above and $v:M\to \C^*$. Since $a$ is real if and only if $b$ is real,
$v$ is real-valued, and then since $v=1$ when $\phi=0$ and
$\mathcal{H}$ is connected, we obtain that $v>0$. This allows to
conclude: $a$ is real with $a_0>0$, if and only if the corresponding
sesquilinear form is definite positive, and then, by
\ref{eq:T_kcov_up_to_mult} and Remark \ref{rem:FS_k-Tkcov}, the proof is complete.
\end{proof}

\section{Approximate geodesics}
\label{sec:gronw-lemma-appr}

It turns out that all the objects on $\mathcal{B}_k$ considered in the
introduction are analytic FIOs close to identity on $(M,J,\omega_0)$, provided the
associated geometric data is close to $\omega_0$. Using the techniques
developed in Section \ref{sec:analyt-four-integr}, we can then prove the
main claims.

\subsection{Geodesics as Fourier Integral Operators}
\label{sec:geodesics-as-fourier}

\begin{prop}\label{prop:Hilbk_are_FIO}
  Let $\phi\in \mathcal{H}$ be analytic and close to $0$. Then
  ${\rm Hilb}_k(\phi)$ is an analytic Fourier Integral Operator close to identity, up to
  $O(e^{-ck})$.
\end{prop}
\begin{proof}
  The proof consists simply in reverting the roles played by
  $\omega_0$ and $\omega_{\phi}$ in Proposition
  \ref{prop:manip_kernel_ops}, and applying this result to the Bergman
  kernel $\Pi_k^{\phi}$, known to be a Fourier Integral Operator for
  $\omega_{\phi}$ by Proposition \ref{prop:calc_Toeplitz}.
\end{proof}

We now turn to what is in fact the technical core of this article,
namely that Bergman rays (i.e. imaginary time Schrödinger propagators)
are, in short time, analytic FIOs close to identity.
\begin{prop}\label{prop:propag_is_FIO}
  Let $f\in C^{\omega}(M,\R)$. Then, for $t\in \C$ small, $e^{tkT_k^{\rm cov}(f)}$ is
  an analytic Fourier Integral Operator close to identity.
\end{prop}
The proof of this Proposition occupies the rest of Subsection
\ref{sec:geodesics-as-fourier}.

We begin by identifying the phase of the FIO. To begin with, let
$\kappa_t$ be the time flow of the Hamiltonian $if$ on
$\widetilde{M}$. In short times, the graph $\mathcal{L}(t)$ of
$\kappa_t$ is a complex Lagrangian, close to identity.

\begin{prop}\label{prop:Lt-BS}
  For $t$ small, the Lagrangian $\mathcal{L}(t)$ is Bohr-Sommerfeld.
\end{prop}
\begin{proof}Since $\mathcal{L}(t)$ is a graph, loops in
$\mathcal{L}(t)$ are of the form $\{(\gamma(s),\overline{\kappa_t(\gamma(s))}),
s\in [0,1]\}$ for some loop $\gamma:[0,1]\to \widetilde{L}$. The
Bohr-Sommerfeld index of this curve is then the Bohr-Sommerfeld index
of $\gamma$ minus that of $\kappa_t(\gamma)$. By Stokes' theorem, this
difference is the integral of $\widetilde{\omega}$ over a surface with
boundaries $\gamma$ and $\kappa_t(\gamma)$. To prove that this
integral is zero, we remark that it is equal to $0$ when $t=0$, and
moreover the derivative of this
integral with respect to $t$ is
$\int_{\kappa_t(\gamma)}\iota_{X_{if}}\omega=i\int_{\kappa_t(\gamma)}\dd
f = 0$.
\end{proof}

Moreover, (in
short time, and restricted to the vicinity of $M$) it satisfies the
group morphism property
\[
  \mathcal{L}(t)\circ \mathcal{L}(s)=\mathcal{L}(t+s).
\]
By Proposition \ref{prop:Lagr-to-sec}, to this family of Lagrangians corresponds a family of holomorphic
sections $\Phi(t)$, with $\Phi(0)=\Psi$. Then, by Proposition
  \ref{prop:compo_FIO}, composition of an
analytic FIO with phase $\Phi(t)$ and one with phase $\Phi(s)$
will be an analytic FIO with phase $\Phi(t+s)$, up to an
exponentially small error, and provided that $t$ and $s$ are small.

Let us introduce a neighborhood $V$ of the diagonal in
$\widetilde{M}$ and a first candidate for the propagator:
\[
  U_0(t)=I_k^{V,\Phi(t)}(1).
\]

The fact that the canonical relation of $U_0(t)$ is the flow of $if$
translates into the following identity.
\begin{lem}\label{prop:identity_principal_symbol}There exists $c>0$ and, for $t$ small, there exists an analytic symbol $g(t)$ (with
  holomorphic dependence on $t$) such that
  \[
    U_0(-t)\left(\frac{\partial U_0(t)}{\partial t}-T_k^{\rm
        cov}(kf)U_0(t)\right)=T_k^{\rm cov}(g(t))+O(e^{-ck}).
    \]
\end{lem}
\begin{proof}
  By Proposition \ref{prop:compo_FIO}, both $k^{-1}U_0(t)\frac{\partial
    U_0(t)}{\partial t}$ and $T_k^{\rm
        cov}(f)U_0(t)$ are analytic covariant Toeplitz operators,
      whose symbol has holomorphic dependence on $t$.
      One only has to show that their principal symbols are equal;
      however, this is true for imaginary time $t$, where
      $\mathcal{L}(t)$ is the complexification of a real
      Lagrangian. Indeed, in this setting, one has a parametrix for
      the propagator at any order
      \cite{zelditch_pointwise_2018,charles_quantum_2020} whose phase
      is exactly $\phi(t)$. Since the principal symbols agree for
      imaginary time, and have holomorphic dependence on $t$, they
      agree everywhere. This concludes the proof.
\end{proof}
Since $\mathcal{L}(t)\circ \mathcal{L}(-t)$ is the diagonal,
Proposition \ref{prop:compo_FIO} gives
\[
  U_0(-t)U_0(t)=T_k^{\rm cov}(b(t))+O(e^{-ck}),
\]
for some analytic symbol $b(t)$ with nonvanishing principal symbol with analytic
dependence on $t$, the same constant $c>0$ as in Lemma
\ref{prop:identity_principal_symbol} and all $t$ small enough.

For $t$ small enough, one has also in operator norm, by
  Proposition \ref{prop:bound-op-norm},
\[
  \|U_0(t)\|+\|(U_0(t))^{-1}\| \leq Ce^{\frac c4 k}
\]
and
\[
  \|e^{tkT_k(f)}\|\leq Ce^{\frac c4 k}.
\]
Hence, for some formal analytic symbols $g(t)$ and $F(t)$,
\begin{align*}
    \frac{\partial}{\partial
  t}\left[e^{-tkT_k(f)}U_0(t)\right]&=e^{-tkT_k(f)}\left[-tkT_k(f)U_0(t)+\partial_tU_0(t)\right]\\
  &=e^{-tkT_k(f)}(U_0(-t))^{-1}T_k^{\rm cov}(g(t))+O(e^{-\frac c2
    k})\\
  &=e^{-tkT_k(f)}U_0(t)(T_k^{\rm cov}(b(t)))^{-1}T_k^{\rm
    cov}(g(t))+O(e^{-\frac c2 k})\\
  &=e^{-tkT_k(f)}U_0(t)T_k^{\rm cov}(F(t))+O(e^{-\frac c2 k}),
\end{align*}
At $t=0$, one has of course $e^{-tkT_k(f)}U_0(t)=T_k^{\rm
  cov}(1)+O(e^{-ck})$. Uniformly for $t$ close to $0$, $F(t)$ belongs
to some analytic class $S^{r,R}_m(V')$ for some neighborhood $V'$ of
the diagonal.

By Proposition \ref{prop:calc_Toeplitz}, we are able to apply the
Picard--Lindelöf theorem to the following (linear) Cauchy problem:
\[
  a'(t)= a(t)\sharp F(t) \qquad \qquad a(0)=1 \qquad \qquad a(t)\in
  S^{2r,2R}_m(V'),
\]
where $\sharp$ denotes the symbol product of covariant
Berezin--Toeplitz operators. There exists a unique solution $a(t)$ to
this Cauchy problem, and one has, for some $c'>0$,
\[
  \frac{\partial}{\partial t}T_k^{\rm cov}(a(t))=T_k^{\rm
    cov}(a(t))T_k(F(t))+O(e^{-c'k}).
  \]
  The true solution of the equation
  \[
    A'(t)=A(t)T_k^{\rm cov}(F(t)) \qquad \qquad A(0)=\Pi_k
  \]
  is uniformly bounded (in operator norm), along with its inverse, as
  $k\to +\infty$. Thus, by the Duhamel formula, one has both
  \begin{align*}
    T_k^{\rm cov}(a(t))&=A(t)+O(e^{-c'k})\\
    e^{-tkT_k(f)}U_0(t)&=A(t)+O(e^{-c'k})
  \end{align*}
  and consequently, for small times,
  \begin{align*}
    e^{tkT_k(f)}&=(T_k^{\rm
    cov}(a(t)))^{-1}U_0(t)+O(e^{-\frac{c'}{2}k})\\
    &=I_k^{V,\varphi(t)}(b(t))+O(e^{-\frac{c'}{2}k})
  \end{align*}
  where we applied again Propositions \ref{prop:compo_FIO} and
  \ref{prop:inv_FIO}. This concludes the proof of Proposition \ref{prop:propag_is_FIO}.
\begin{rem}
  The ability to find an analytic symbol in the propagator for
  Proposition \ref{prop:propag_is_FIO} relies on an application of the Picard-Lindelöf
  theorem on a space of analytic symbols; it is essential that
  multiplication by a more regular symbol leaves invariant an analytic class.
\end{rem}

\subsection{Geodesic equations}
\label{sec:geodesic-equations}
To conclude the proof of Theorem \ref{thr:FIOs}, it remains to show
that the phases of the Fourier Integral Operators appearing in
Propositions \ref{prop:Hilbk_are_FIO} and \ref{prop:propag_is_FIO} are
identical whenever $\phi(t)$ is a real-analytic Mabuchi geodesic with
$\phi(0)=0$ and $f=\dot{\phi}(0)$.

\begin{prop}\label{prop:same_phase}
  Let $V$ be a neighbourhood of the diagonal in $M\times
    \overline{M}$. Let $\dot{\phi}(0)\in C^{\omega}(M)$. Let $\phi(t)$ be the
 Mabuchi geodesic with initial value $(0,\dot{\phi}(0))$ and let $\Psi(t),a(t)$ be respectively a
  holomorphic section of $L\boxtimes \overline{L}$ over $V$ and a holomorphic
  analytic symbol on $V$ such that, as in Proposition
  \ref{prop:Hilbk_are_FIO},
  \[
    {\rm Hilb}_k(\phi(t))=I_k^{V,\Psi(t)}(a(t))+\O(e^{-ck})
  \]
  for some $c>0$, uniformly for $t$ in a neighbourhood of $0$.

  Let also $\Phi(t),b(t)$ be respectively a
  holomorphic section of $L\boxtimes \overline{L}$ over $V$ and a holomorphic
  analytic symbol on $V$ such that, as in Proposition
  \ref{prop:propag_is_FIO},
  \[
    \exp(tkT_k^{\rm cov}(-\dot{\phi}(0)))=I_k^{V,\Phi(t)}(b(t))+\O(e^{-ck})
  \]
  for some $c>0$, uniformly for $t$ in a neighbourhood of $0$.

  Then $\Phi=\Psi$.

\end{prop}
\begin{proof}
    Letting $A(t)=\exp(tkT_k^{\rm cov}(-\dot{\phi}(0)))$, then
    $t\mapsto A(t)$ is the representation as ${\rm
      Hilb}_k(0)$-self-adjoint operators of a geodesic $\gamma_k(t)$ in
    $\mathcal{B}_k$. Representants $A^{\psi_{\rm ref}}(t)$ of this geodesic at any base point
    ${\rm Hilb}_k(\psi_{\rm ref})$ satisfy
    \begin{equation}\label{eq:geod_based_psi}
      \ddot{A}^{\psi_{\rm ref}}(t)=\dot{A}^{\psi_{\rm ref}}(t)(A^{\psi_{\rm
          ref}}(t))^{-1}\dot{A}^{\psi_{\rm ref}}(t).
    \end{equation}
    By Proposition \ref{prop:manip_kernel_ops} and \ref{prop:propag_is_FIO}, if $\psi_{\rm ref}$ is real-analytic and close to $0$
    then $A^{\psi_{\rm ref}}$ and its time derivatives are analytic
    Fourier Integral Operators close to identity.

    Moreover, for any $t\in \R$ close to $0$, since $\exp(tkT_k^{\rm
      cov}(-\dot{\phi}(0)))$ is self-adjoint, then $\Phi(t)$ is
    self-adjoint, meaning that
    \[
      \overline{\Phi(t,x,y)}=\Phi(t,y,x).
    \]
    Following Proposition \ref{prop:FSk_of_FIO}, there exists a path
    of elements of $\mathcal{H}$, denoted $t\mapsto \phi_1(t)$, such
    that, for every $t$,
    \[
      A^{\phi_1(t)}(t)=T_k^{\rm cov,\phi_1(t)}(g(t))
    \]
    for some real-analytic symbol $g(t)$. More generally, one can
    write
    \[
      A^{\phi_1(T)}(t)=I_k^{\Phi_t^T,V}(g(t,T))
    \]
    where $\Phi_T^T$ is the phase of the Bergman kernel associated
    with $\phi_1(T)$ and $g$ is a real-analytic symbol.

    By successive differentiation, we obtain
    \begin{align*}
      \frac{\dd}{\dd t}A^{\phi(T)}(t)|_{t=T}&=kT_k^{\rm
                                              cov,\phi_1(T)}(\tfrac{\dd
                                              \log \Phi^T_t}{\dd
                                              t}g|_{t=T}+k^{-1}\tfrac{\dd
                                              g}{\dd t}|_{t=T})\\
      \frac{\dd^2}{\dd t^2}A^{\phi(T)}(t)|_{t=T}&=k^2T_k^{\rm cov,\phi_1(T)}[(\tfrac{\dd
                                                  \log \Phi^T_t}{\dd
                                                  t}g|_{t=T})^2+k^{-1}\tfrac{\dd^2
                                                  \log \Phi^T_t}{\dd
                                                  t^2}g|_{t=T}+2k^{-1}\tfrac{\dd
                                                  \log \Phi^T_t}{\dd
                                                  t}|_{t=T}\tfrac{\dd
                                                  g}{\dd t}|_{t=T}+k^{-2}\tfrac{\dd^2
                                                  g}{\dd t^2}|_{t=T}].
    \end{align*}
   Applying now the geodesics equation \eqref{eq:geod_based_psi} and
   the subprincipal calculus identity \eqref{eq:Charles_1}, we obtain
   that
   \[
     \left. \tfrac{\dd^2 \log \Phi_t^T}{\dd t^2}\right|_{t=T}= \left\|\partial \tfrac{\dd
        \log \Phi^T_t}{\dd
        t}|_{t=T}\right\|_{\phi_1(T)}^2.
  \]
  Applying \eqref{eq:var_phase}, and Proposition
  \ref{prop:calc_Toeplitz}, there holds $ \frac{\dd \log \Phi^T_t}{\dd
    t}|_{t=T}=-\dot{\phi}_1(t)$, and then by
  \eqref{eq:var2_phase},
  \[
    \frac{\dd^2\phi_1}{\dd t^2}=-\left.\frac{\dd^2\Phi^T_t}{\dd
        t^2}\right|_{t=T}+2\|\partial \dot{\phi}_1\|_{\phi_1}^2.
  \]
  Thus $\phi_1$ satisfies the equation
  \[
    \ddot{\phi_1}(t)= \|\partial\dot{\phi_1}\|_{\phi_1}
  \]
  which is exactly the Mabuchi geodesic equation \eqref{eq:Mabuchi_geod_eq}. At $t=0$ one has
  $\frac{\dd A^{\psi_0}}{\dd t}=kT_k^{\rm cov}(-\dot{\phi}(0))$ so
  that $\dot{\phi}_1(0)=\dot{\phi}(0)$. By the Cauchy-Kovalevskaya
  theorem, one can conclude that $\phi_1=\phi$, and therefore by
  Proposition \ref{prop:manip_kernel_ops}, that $\Psi(t)=\Phi(t)$.
\end{proof}

We are now also in position to complete the proof of Theorem \ref{thr:closeness_geodesics_IVP}.

  \begin{prop}\label{prop:short_time}
    Given a real-analytic Kähler manifold $(M,J,\omega_0)$and given $r>0,C_0>0,m\in \R$,
    there exists $\epsilon>0$ and $C_1>0$ such that, for all
    $(\phi_1,\dot{\phi}_1)\in T\mathcal{H}$ such that
    \[
      \|\phi_1\|_{H(r,m)}+\|\dot{\phi}_1\|_{H(r,m)}\leq C_0
    \]
    the Mabuchi geodesic equation with initial conditions
    $(\phi_1,\dot{\phi}_1)$ has a solution $\phi(t)$ as a real-analytic curve on
    times $(-\epsilon,\epsilon)$. Moreover, for every $k\in \N$, the
    geodesic $\gamma_k$ on $\mathcal{B}_k$ whose initial condition is $({\rm
      Hilb}_k(\phi_1),\dd {\rm Hilb}_k(\dot{\phi}_1))$ satisfies, for
    all $t\in (-\epsilon,\epsilon)$,
    \[
      \dist_{\mathcal{B}_k}(\gamma_k(t),{\rm Hilb}_k(\phi(t)))\leq
      C_1k^{\frac d2}
    \]
    and
    \[
      \dist_{\mathcal{H}}({\rm FS}_k(\gamma_k(t)),\phi(t))\leq
      C_1k^{-1}.
      \]
    \end{prop}
    \begin{proof}
Let $\phi(t)$ be a geodesic in Mabuchi space, and let
$H(t)$ be the geodesic associated with the projected initial value problem. Then there
exists $\Phi$ and two analytic symbols $a$ and $b$ such that
\begin{align*}
  {\rm Hilb}_k(\phi(t)):(x,y)& \mapsto k^d\Phi^{\otimes
  k}(x,y)a(t,x,y;k^{-1})+O(e^{-ck})\\
  H(t):(x,y)&\mapsto  k^d\Phi^{\otimes
  k}(x,y)b(t,x,y;k^{-1})+O(e^{-ck}).
\end{align*}

In particular, by Proposition \ref{prop:compo_FIO},
\[
  {\rm Hilb}_k(\phi(t))(H(t))^{-1}=T_k^{\rm cov}(r)+O(e^{-c'k}),
\]
where $r$ is an analytic symbol. Hence
\[
  \dist_{\mathcal{B}_k}({\rm
    Hilb}_k(\phi(t)),H(t))=\dist_{\mathcal{B}_k}(I,T_k^{\rm
    cov}(r)+O(e^{-c'k}))=O(k^{\frac d2}).
\]
This concludes the proof of \eqref{eq:close-geods-Bk}.

We now turn to the proof of \eqref{eq:close-geods-FSk}. By
Proposition \ref{prop:Tian-Zelditch},
\[
  {\rm FS}_k({\rm
    Hilb}_k(\phi(t)))=\phi(t)+O_{C^{\infty}}(k^{-1}).
\]
Moreover, since $H(t)$ is a Fourier Integral Operator with the same
phase, by Proposition \ref{prop:FSk_of_FIO} its ${\rm FS}_k$ is the same up
to $O_{C^{\infty}}(k^{-1})$. Noticing that
$\|\phi_1-\phi_2\|_{C^{2}}\gtrsim
\dist_{\mathcal{H}}(\phi_1,\phi_2)$, the proof is complete.
\end{proof}


\section{No analytic geodesic between analytic endpoints}
\label{sec:geom-space-quant}

To conclude this article, we argue for the \emph{non-existence} of a
solution to the boundary value problem in real-analytic regularity,
that is, Conjecture \ref{conj:no_analytic_bd_value}.

Let $E,E'$ be two analytic function spaces. Let $(M,J,\phi_0)$ be a Kähler metric in $E$. Let
$\phi_1\in \mathcal{H}\cap E$ and suppose that $\phi_1$ is
joined to $\phi_0$ by an analytic geodesic $\phi_t$ in $E'$ such that
$\dot{\phi}_t\in E'$.

There exists $\epsilon>0$ such that, for all $t_0\in [0,1]$ and all $f\in B_{E'}(0,\epsilon)$, the
conclusion of Theorem \ref{thr:closeness_geodesics_IVP}
apply up to time $1$ for the Cauchy problem with initial data
$(\phi_t,f)$. Thus, there are well-defined complex Lagrangians
$\mathcal{L}_{t_0}(t-t_0)$ corresponding to the change from $\phi_{t_0}$ to $\phi_t$, for all
$|t-t_0|<\frac{\epsilon}{R}$, where $R=\sup_{t\in
  [0,1]}\|\dot{\phi}_t\|_E'$.

The imaginary time Lagrangians $\mathcal{L}_{t_0}(i\tau)$ can be extended to
$t\in \R$; they are the graphs of the flow of the
time-independent Hamiltonian $\dot{\phi}_{t_0}$. In this respect, they form a closed
set with empty interior amongst the graph of flows of time-dependent
Hamiltonians in $E'$ (this classical result goes as follows: graphs of
flows of time-dependent Hamiltonians generically intersect each other
cleanly, so that their periodic points form a set of dimension $0$;
however periodic points of $\mathcal{L}_{t_0}(i\tau)$ correspond
to closed orbits of $\dot{\phi}_{t_0}$, which form a set of dimension $1$).

Close to $\phi_{t_0}$ in the norm of $E'$, the map which to a Kähler potential associates its
Lagrangian should be a well-defined diffeomorphism. Hence, there exists
$\delta>0$ such that, for all
$t_0$, once $\phi_{t_0}$ is fixed, the functions $\phi_{t_0+i\tau}$, for $|\tau|<2\delta$, belong to a
closed set with empty interior, sitting in the totally real submanifold of
Kähler potentials whose Lagrangians (with respect to $\phi_{t_0}$) are real.

One should be able to conclude that the functions $\phi_{t_0+\tau}$, for
$|\tau|<2\delta$, belong to a closed set with empty interior among
real-valued Kähler potentials. To prove this, the missing link is a
continuous map, near $\phi_{t_0}$, between analytic real-valued Kähler
potentials and Kähler potentials whose Lagrangians are real, and which
sends solutions of the Mabuchi geodesic equation to autonomous
flows. A candidate for such a map, using Berezin--Toeplitz
quantization, is $\lim_{k\to +\infty} {\rm FS}_k\circ\mathcal{J}\circ
{\rm Hilb}_k$, where $\mathcal{J}:SL(d_k,\C)/SU(d_k,\C)\to
SU(d_k,\C)$ is the exponentiation of the multiplication by $i$ between
tangent spaces at ${\rm Hilb}_k(\phi(t_0))$. By Theorem \ref{thr:closeness_geodesics_IVP}, this map is indeed well-defined at the end-point
of an analytic geodesic, but the behaviour of the limit for general
closeby analytic potentials is not clear.

Non-genericity for short times would be enough to prove the
conjecture: if we can prove progressively that the family
$(\phi_{\delta},\phi_{2\delta},\ldots,\phi_{\lfloor
  \delta^{-1}\rfloor\delta},\phi_1)$, belongs to a closed set with
empty interior, then $\phi_1$ itself satisfies this property.

\section{Acknowledgements}
\label{sec:acknowledgements}
This work was supported by the National Science Foundation under Grant No. DMS-1440140 while A. Deleporte was in
residence at the Mathematical Sciences Research Institute in Berkeley,
California, during the Fall 2019 semester, by the Centre National de
la Recherche Scientifique under a PEPS JCJC grant during 2021, as well
as by the NSF grant no. DMS-1810747 for the whole duration of this project.



\bibliographystyle{abbrv}
\bibliography{main}
\end{document}